\documentclass[preprint,12pt]{article}

\usepackage{pgf,tikz}
\usetikzlibrary{arrows}
\usepackage{array,multirow,makecell}
\setcellgapes{1pt} \makegapedcells
\usepackage{amssymb}
\usepackage{subfig}
\usepackage{graphicx}
\usepackage{here}
\usepackage{epsfig}
\usepackage{color}
\usepackage{pgf,tikz,amsthm,amsmath}
\usetikzlibrary{arrows}
\usepackage{bbold}
\usepackage{epstopdf}
\usepackage{soul}

\usepackage{hyperref}
\hypersetup{
    colorlinks=true,
    linkcolor=blue,
    filecolor=magenta,      
    urlcolor=cyan,
    citecolor=red
}

\newcommand{\rr}{\mathcal{R}}

\newcommand{\A}{\mathcal{A}}

\newcommand{\tcm}{\textcolor{black}}

\newcommand{\tcb}{\textcolor{black}}
\newcommand{\comment}[1]{}

\newtheorem{theorem}{\textbf{Theorem}}
\newtheorem{lemma}{\textbf{Lemma}}
\newtheorem{proposition}{\textbf{Proposition}}
\newtheorem{definition}{\textbf{Definition}}

\newcommand{\R}{\mathbb{R}}
\begin{document}

%% Title, authors and addresses

%% use the tnoteref command within \title for footnotes;
%% use the tnotetext command for theassociated footnote;
%% use the fnref command within \author or \address for footnotes;
%% use the fntext command for theassociated footnote;
%% use the corref command within \author for corresponding author footnotes;
%% use the cortext command for theassociated footnote;
%% use the ead command for the email address,
%% and the form \ead[url] for the home page:
%% \title{Title\tnoteref{label1}}
%% \tnotetext[label1]{}
%% \author{Name\corref{cor1}\fnref{label2}}
%% \ead{email address}
%% \ead[url]{home page}
%% \fntext[label2]{}
%% \cortext[cor1]{}
%% \address{Address\fnref{label3}}
%% \fntext[label3]{}

%\title{A PDE framework for the SIT issue}
\title{On the use of Traveling Waves for Pest/Vector elimination using the Sterile Insect Technique}
\author{R. Anguelov$^{1}$,  Y. Dumont$^{1,2,3}$ \footnote{corresponding author: yves.dumont@cirad.fr}, and I.V. Yatat Djeumen$^4$ \\
{\small $^1$ University of Pretoria, Department of Mathematics and Applied Mathematics, South Africa }\\
\small  $^2$ CIRAD, UMR AMAP, P\^ole de Protection des Plantes, F-97410 Saint-Pierre, La R\'eunion, France \\
\small  $^3$ AMAP, Univ Montpellier, CIRAD, CNRS, INRAE, IRD, Montpellier, France \\
\small  $^4$ University of Yaound\'e I, National Advanced School of Engineering  of Yaound\'e, Cameroon}

\maketitle

\begin{abstract}
The development of sustainable vector/pest control methods is of utmost importance to reduce the risk of vector-borne diseases and pest damages on crops. Among them, the Sterile Insect Technique (SIT) is a very promising one. In this paper, using diffusion operators, we extend a temporal SIT model, developed in a recent paper, into a partially degenerate reaction-diffusion SIT model. Adapting some theoretical results on traveling wave solutions for partially degenerate reaction-diffusion equations, we show the existence of mono-stable and bi-stable traveling-wave solutions for our SIT system. The dynamics of our system is driven by a SIT-threshold number above which the SIT control becomes effective and \tcm{drives} the system to elimination, using massive releases. When the amount of sterile males is lower than the \tcm{SIT-threshold}, the SIT model experiences a strong Allee effect such that a bi-stable traveling wave \tcm{solution} can exist and can also be used to derive an effective long term strategy, mixing massive and small releases. We illustrate some of our theoretical results with numerical simulations, and, also explore numerically spatial-localized SIT control strategies, using massive and small releases. We show that this "corridor" strategy can be efficient to block an invasion and eventually can be used to push back the front of a vector/pest invasion.
\end{abstract}

\vspace{1cm}

{ \bf keywords}: Sterile insect technique; Vector control; Pest control; Partially degenerate reaction-diffusion system; Allee effect; Traveling wave; Corridor strategy

\section{Introduction}
Food security and Health security have become of utmost importance around the world because pests and diseases vectors can travel, invade, and settle in new areas causing crop losses and diseases epidemics or pandemics. For instance, according to WHO, 3.9 billons of people are at risk of infection with dengue viruses \cite{Back2013} and modeling estimate indicates that between 284–528 millions of people are infected per year. Dengue is considered as a tropical disease, but as its vectors, like \textit{Aedes albopictus} \cite{Gratz2004}, are now established in the South of Europe (continuing to spread northward), and also in North-America, the risk of epidemics is real.  Similarly, fruit flies, once established, may cause $25$ to $50$ percent losses in food-crop harvests on a very wide range of crops. Among them, the oriental fruit fly, \textit{B. dorsalis}, is the most damaging one (see \cite{Vargas2015} for an overview) . Native from Asia, it invaded Africa from 2004, and La R\'eunion island in 2017. Few individuals have been recorded in Italy in May 2018. Again, being highly invasive, the risk for southern Europa is high.

That is why pest/vector control is absolutely necessary. In the past, most of the control strategies relied on chemical control. Now, we know that \tcm{this is} not sustainable, as chemicals have \tcm{negative} environmental effects, with also the risk of vector/pest resistance appearance in addition to toxic effects on human health, such that most of the chemical cannot be used and are not efficient anymore.
% For example, in Martinique, located in the French West Indies, massive spraying of \textit{Deltamethrin} have led to increasing the resistance of the mosquito. Precisely, ca. 60\% of \textit{Aedes aegypti}, the vector of Dengue disease, were resistant to the \textit{Deltamethrin} \cite{Dumont2012}. Hence, in 2010, in Martinique and Guadeloupe, two French overseas departments, it was no longer possible to use adulticide against \textit{Aedes aegypti} \cite{Dumont2012}. 
\tcm{Therefore,} environmental-friendly vector control and pest management strategies have received widespread attention and have become challenging issues in order to reduce or prevent devastating impact on health, economy, food security, and biodiversity.%, like Malaria, Dengue, Chikungunya \cite{Dumont2012,Anguelov2013,Seirin2013,Bliman2019,Anguelov2019} or crop pests, like fruit  flies \cite{Mike2018,Aronna2020}. 

One sustainable, environmental-friendly and promising alternative for vector/pest control is the Sterile Insect Technique (SIT). This is an old control technique, proposed in the 30s and 40s by three key researchers in the USSR, Tanzania and the USA and, first, applied in the field in the 50's \cite{SIT}. SIT consists to sterilize male pupae using ionizing irradiation and to release a large number of these irradiated males such that they will mate wild females, that will have no viable offspring. Hence, it will result in a progressive decay of the targeted population \cite{Anguelov2019, SIT, Knipling1955, Seirin2013}. For
mosquitoes, other sterilization techniques have been developed using either genetics (the Release of Insects carrying a Dominant Lethal, in short
RIDL) (see \cite{Seirin2013,Seirin2013b} and reference therein), or Cytoplasmic Incompatibility using a bacteria (Wolbachia) \cite{Sinkins2004,Strugarek2019}. For fruit flies, ionizing radiation has been used so far \cite{SIT}. However a genetically engineered Medfly (\textit{C. capitata}) has been developed and tested \cite{MedflyRIDL}. Even if, conceptually, the Sterile Insect Technique appears to be very simple, in reality it is not and the process to reach field applications is long and complex \cite{SIT}. That is why, even if SIT is now used routinely in some places around the world (in Spain or Mexico against the mediterranean fruit flies, for instance), they are still many SIT feasibility projects around the World, including three in France, against \textit{Aedes albopictus}, vector of Dengue, Chikungunya and Zika in La R\'eunion, against the fruit fly \textit{Ceratitis capitata} in Corsica, and the fruit fly \textit{Bactrocera dorsalis} in La R\'eunion \cite{Aronna2020}.

We firmly believe that mathematical modeling and computer simulations can be additional and efficient tools within these ongoing programs in order to prevent SIT failures, improve field protocols, and test assumptions before any field investigations, etc.

This work is an extension of \cite{Anguelov2019} where only  a mean-field temporal SIT mathematical model was developed and studied aiming to assess the SIT potential as a long term control tool for vector/pest population reduction or elimination, combining massive and small releases. In the present contribution, we take into account vector/pest adult's dispersal, keeping a certain genericity which allows to apply our spatio-temporal model for several vectors or pests, like mosquitoes or fruit flies. 

%We propose a (spatial) control which relies on an Allee effect generated by the SIT control. Indeed, it has been shown that even low levels of SIT control produce a tangible minimum survival density, below which the population declines to extinction, see e.g. \cite{Anguelov2019} and references therein. 

%Our work stands within the framework of two SIT feasibility projects in La Réunion (France) and one SIT field project in Corsica (France). In La R\'eunion, the first project, the TIS 2B project, is against \textit{Aedes albopictus}; it is funded by the French Ministry of Health and the European Regional Development Fund (ERDF); the second project, named GEMDOTIS, is against a very damaging fruit fly, \textit{Bactrocera dorsalis}, that first appeared in La Réunion in 2017. It is funded by the French government, through the 2018 EcoPhyto Call. La Réunion is a French tropical island, located in the Indian Ocean, East from Madagascar and South-West from Mauritius island. The third project, called CeraTIS-Corse and funded by the French government through the 2019 EcoPhyto Call, focus on \textit{Ceratis capitata}, the mediterranean fruit fly, that impacts many fruits orchards in Corsica, a french department located in the Mediterranean sea, southeast of France and west of Italy.

Very few works exist on SIT taking into account explicitly the spatial component. First because, from the ecological point of view, knowledge are scarce and incomplete. This last point might be strange, but in general it is much more difficult to study the behavior of pest or vector in the field, and, \tcb{very often}, many studies only rely on laboratory or semi-field studies. %To illustrate our point of view: very little is known about the interaction between mosquito and vegetation, i.e. what is the impact of vegetation in the spreading of vectors. 
Also, from the modeling point of view, spatio-temporal models are more difficult to develop and they require more sophisticated tools to be studied. However, some attempts have been made in order to have some insights in SIT systems.
In \cite{Manoranjan1986}, the authors were the first to consider a reaction-diffusion equation to take into account the spreading of a pest in a SIT model. This work was completed in \cite{Lewis1993}, where the release of sterile females was also considered. Tyson et al. \cite{Tyson2007} used an advection-reaction-diffusion model to study SIT against codling moth in pome fruit orchards. In \cite{Marsula1994}, the authors consider diffusion like in \cite{Manoranjan1986}, with a time discrete SIT model, to study a barrier strategy. Similarly, Serin Lee et al. \cite{Seirin2013, Seirin2013b} studied SIT control with barrier effect using a system of two reaction-diffusion equations for the wild and the sterile populations. In \cite{Ferreira2008}, the authors consider discrete cellular automata and show that SIT can fail when oviposition containers distribution is too heterogeneous. A recent work \cite{Bishop2014} includes impulsive SIT releases.
A more complex 2D spatial mosquito model, using a system of coupled ordinary differential equations and advection-reaction-diffusion equations, with applications on SIT, was studied by Dufourd and Dumont \cite{Dufourd2013} highlighting the importance of environmental parameters, like wind, in SIT release strategies. However, no theoretical results were obtained and the results mainly rely on numerical simulations. 

The rest of the paper is organized as follows. Section \ref{preliminay} is devoted to recall preliminaries, including theoretical results obtained in \cite{Anguelov2019}, that are helpful for our current study. Section \ref{spatio-temporal-SIT-model} deals with the formulation and the study of the spatio-temporal SIT model. We also extend results previously obtained by Fang and Zhao \cite{Fang2009} to show existence of monostable and bistable traveling wave solutions for partially degenerate reaction-diffusion equations.  In both cases, the wave solution involves the elimination equilibrium or the zero equilibrium and a positive equilibrium. Section \ref{Numerical-simulations} deals with numerical simulations in order to support the  theoretical results and also go further. In particular, we consider a strategy developed and studied in \cite{Anguelov2019} where massive and small releases were considered to drive a wild population to elimination. Here we extend this strategy using spatially-localized massive releases, within a given (spatial) corridor, coupled with small releases in the pest/vector free area.
%for two control scenarios: an endemic case and an emerging outbreak for a  vector/pest. In the former case, a vector/pest is endemic while in the latter case invading vector/pest establish and cause a local outbreak in a previously vector/pest-free region. Precisely, we discuss how the size of the SIT release and the length of the region in which sterile insects are released may successfully $(i)$ recover a vector/pest-free sub-region in the endemic case; $(ii)$ stop the vector/pest invading wave, thus, maintains a vector/pest-free sub-region and, $(iii)$ push back the vector/pest invading wave in order to obtain/recover more vector/pest-free sub-regions. 
Finally, in section \ref{Conclusion}, we summarize the main results and provide future ways to improve or extend this work.

\section{Preliminaries}\label{preliminay}
Let us first recall some notations that will be used in this work. Let $\mathcal{C}$ be the set of all bounded and
continuous functions
 from $\mathbb{R}$ to $\mathbb{R}^n$. For $u=(u_1,...,u_n)'$,
 $v=(v_1,...,v_n)'\in\mathcal{C}$, we define $u\geq v$ (resp. $u\gg
 v$) to mean that $u_i(x)\geq v_i(x)$ (resp. $u_i(x)>v_i(x)$), $1\leq i\leq
 n$, $\forall x\in\mathbb{R}$, and $u>v$ to mean that $u\geq v$ but
 $u\neq v$. Any vector of $\mathbb{R}^n$ can be identified as an element in $\mathcal{C}$.  For any $r\in\mathbb{R}$, we use boldface
\textbf{r} to denote the vector with each component being $r$, i.e.,
\textbf{r} = $(r,..., r)'$. Moreover, for a square matrix $A$, its
stability modulus is defined by
$$s(A):=\max\{Re\lambda:\det(\lambda I-A)=0\}.$$

\begin{definition}(Irreducible matrix, \cite[page 56]{Smith1995})\label{definition-Smith}\\ An $n\times n$ matrix $A =
(a_{ij})_{1\leq i,j\leq n}$ is irreducible if for every nonempty,
proper subset $I$ of the set $N = \{1,2,... , n\}$, there is an
$i\in I$ and $j \in J = N\setminus I$ such that $a_{ij}\neq0$.
\end{definition}

Let $M_T(t)$ be the number of sterile insects at time $t$, $1/\mu_T$ the average lifespan of sterile insects and $\Lambda$ the number of sterile insects released per
unit of time. The dynamics of $M_T$ is described by 
$$ \displaystyle\frac{dM_T}{d t} = \Lambda-\mu_TM_T.$$
Assuming $t$ large enough, we may assume that $M_T(t)$
is at its equilibrium value $M_T=\Lambda/\mu_T$.  Following \cite{Anguelov2019},
 the minimalistic SIT model is defined as follows

\begin{equation}\label{SIT-ode}
 \left\{%
\begin{array}{lcl}
 \displaystyle\frac{dA}{d t} &=& \phi F-(\gamma+\mu_{A,1}+\mu_{A,2}A)A,\\
     \displaystyle\frac{dM}{d t} &=& (1-r)\gamma A-\mu_MM,\\
    \displaystyle\frac{dF}{d t} &=& \displaystyle\frac{ M}{M+M_T}r\gamma A-\mu_FF,\\
\end{array}
\right.
\end{equation}
where parameters and state variables are described in Table
\ref{table-description}, page \pageref{table-description}. Note also that system (\ref{SIT-ode}) is monotone cooperative \cite{Smith1995}.

\begin{table}[H]
	\centering
	\begin{tabular}{|l|l|}
		\hline
		% after \\: \hline or \cline{col1-col2} \cline{col3-col4} ...
		Symbol & Description \\
		\hline
		$A$ & Immature stage (gathering eggs, larvae, nymph or pupae stages) \\
		\hline
		%$Y$ & Young female mosquitoes \\
		%    \hline
		$F$ & Fertilized and eggs-laying females \\
		\hline
		$M$ & Males \\
		\hline\hline
		$\phi$ & Number of eggs at each deposit per capita (per day) \\
		\hline
		$\gamma$ & Maturation rate from larvae to adult (per day) \\
		\hline
		$\mu_{A,1}$ & Density independent mortality rate of the aquatic stage (per day) \\
		\hline
		$\mu_{A,2}$ & Density dependent mortality rate of the aquatic stage (per day$\times$ number) \\
		\hline
		$r$ & Sex ratio \\
		\hline
		$1/\mu_F$ & Average lifespan of female (in days) \\
		\hline
		$1/\mu_M$ & Average lifespan of male (in days) \\
		\hline
	\end{tabular}
	\caption{Description of state variables and parameters of model (\ref{SIT-ode})}\label{table-description}
\end{table}

\noindent The basic offspring number related to system (\ref{SIT-ode}) is defined as follows
$$\rr=\displaystyle\frac{r\gamma\phi}{\mu_F(\gamma+\mu_{A,1})}.$$
%\subsection{Equilibria of the SIT model (\ref{SIT-ode})}
When $\rr >1$, we set
\begin{equation}\label{MT1}
  Q=
  \displaystyle\frac{\mu_{A,2}\mu_M}{(\gamma+\mu_{A,1})(1-r)\gamma}  \quad \mbox{and}\quad  M_{T_{1}} = \displaystyle\frac{(\sqrt{\rr}-1)^2}{Q},
\end{equation}
where $M_{T_1}$ is the SIT-threshold above which the wild population is driven to elimination.
In \cite{Anguelov2019}, when $M_T=0$, we proved that model (\ref{SIT-ode}) admits only the elimination equilibrium, $\bf{0}$ when $\rr\leq1$, and a unique positive equilibrium $E^*=(A^*,M^*,F^*)'$, the wild equilibrium, in addition to the elimination equilibrium, whenever $\rr>1$. Then, assuming $\rr>1$, we showed that when $M_T\in(0,M_{T_1})$ then model (\ref{SIT-ode}) has
two positive equilibria $E_{1,2}=(A_{1,2},M_{1,2},F_{1,2})'$
with $E_1<E_2$, namely
\begin{equation}\label{Equi_1and2}
\begin{array}{ccl}
  A_{1,2} & = & \displaystyle\frac{\mu_M}{(1-r)\gamma}M_{1,2}, \\
  F_{1,2} & = & \displaystyle\frac{(\gamma+\mu_{A,1}+\mu_{A,2}A_{1,2})A_{1,2}}{\phi}, \\
  M_{1} & = & \displaystyle\frac{M_T^*}{\alpha_{+}},\\
  M_{2} & = & \displaystyle\frac{M_T^*}{\alpha_{-}}.
\end{array}
\end{equation}
with 
\begin{equation}
\begin{array}{l}
    \Delta (M_T^*)= ((\sqrt{\rr}-1)^2-M_T^*Q)((\sqrt{\rr}+1)^2-M_T^*Q), \\
    \alpha_{\pm}=\displaystyle\frac{(\rr-1-QM_T^*)\pm\sqrt{\Delta(M^*_T)}}{2}.
    \end{array}
\end{equation} 
When $M_T=M_{T_1}$ then the two positive equilibria $E_{1,2}$ collide into a single equilibrium $E_\dagger$. The asymptotic behavior of model \eqref{SIT-ode} is summarized in the next theorem.

\begin{theorem}\cite{Anguelov2019}\\\label{SIT-ode-theorem}
Assume $\rr>1$. System (\ref{SIT-ode}) defines a dynamical system on $D=\R^3_+$ for
any $M_T\in[0,+\infty)$. Moreover,
\begin{enumerate}
    \item When $M_T>M_{T_1}$ then equilibrium \textbf{\emph{0}} is
    globally asymptotically stable for system (\ref{SIT-ode}).
    \item When $M_T=M_{T_1}$ then system (\ref{SIT-ode}) has two
    equilibria \textbf{\emph{0}} and $E_\dag$ with $\textbf{\emph{0}}<E_\dag$. The set $\{x\in \mathbb{R}^3:\textbf{\emph{0}}\leq
    x< E_\dag\}$ is in the basin of attraction of \textbf{\emph{0}}
    while the set $\{x\in \mathbb{R}^3: x\geq E_\dag\}$ is in the basin of attraction
    of $E_\dag$.
    \item When $0<M_T<M_{T_1}$ then system (\ref{SIT-ode}) has three
    equilibria \textbf{\emph{0}}, $E_1$ and $E_2$ with $\textbf{0}<E_1<E_2$. The set $\{x\in \mathbb{R}^3:\textbf{\emph{0}}\leq
    x< E_1\}$ is in the basin of attraction of \textbf{\emph{0}} while the set $\{x\in \mathbb{R}^3: x> E_1\}$ is in the basin of attraction
    of $E_2$.
\end{enumerate}
\end{theorem}

Based on the previous theorem, it is clear that SIT always needs to be maintained. In \cite{Anguelov2019}, using the strong Allee effect induced by SIT, we have developed a long term sustainable strategy using, first, massive releases, and then, small releases. Indeed, using massive releases, i.e. $M_T > M_{T_1}$, we drive the population into the set $\{x\in \mathbb{R}^3:\textbf{0}\leq
    x< E_1 (M_T)\}$ for a given (small) value for $M_T<<M_{T_1}$, in a finite time (see Theorem 4 in \cite{Anguelov2019}). Then, the control goes on with only small releases (see Theorem 5 in \cite{Anguelov2019}). This strategy allows to use a limited number of sterile insects and also to treat large area, step by step. 

Now, we take into account the spatial dynamics of the insects. The main objective is to show how a bi-stable traveling wave, generated by the releases of sterile males, can be helpful to control a wild mosquito/pest invasion.
%Therefore, we consider the same conceptual framework as in the space-implicit scheme. We assume that, without SIT control, wild insects are distributed according to their homogeneous wild equilibrium $E^*$ and spread in the spatial domain toward insects-free areas through a mono-stable traveling wave solution. The issue, now, is to what extend the SIT release ($M_T$) that generates a bistable wave can help to control (stop) this bistable wave and even reverse its orientation.

\section{A spatio-temporal SIT model}\label{spatio-temporal-SIT-model}
Taking into account adult vectors or pests dispersal through Laplace operators, model (\ref{SIT-ode}) becomes the following partially degenerate reaction-diffusion system:

\begin{equation}\label{SIT-pde}
 \left\{%
\begin{array}{lcl}
 \displaystyle\frac{\partial A}{\partial t} &=& \phi F-(\gamma+\mu_{A,1}+\mu_{A,2}A)A, \quad (t,x)\in\mathbb{R}_+\times\mathbb{R}\\
 %\displaystyle\frac{\partial Y}{\partial t} &=& d_Y\displaystyle\frac{\partial^2 Y}{\partial x^2}+ r\gamma A-(+\mu_Y)Y,\\
    \displaystyle\frac{\partial M}{\partial t} &=& d_M\displaystyle\frac{\partial^2 M}{\partial x^2}+ (1-r)\gamma A-\mu_MM,\\
    \displaystyle\frac{\partial F}{\partial t} &=& d_F\displaystyle\frac{\partial^2 F}{\partial x^2}+\displaystyle\frac{ M}{M+M_T}r\gamma A-\mu_FF,\\
\end{array}
\right.
\end{equation}
where $d_F$ and $d_M$ denote fertilized females and males diffusion
rate, respectively. In addition, system (\ref{SIT-pde}) is considered
with non-negative and sufficiently smooth initial data. We will now address the question of existence of mono-stable and
bistable traveling wave solutions in system (\ref{SIT-pde}). Unfortunately, in our case, we cannot apply directly the results from Fang and Zhao \cite{Fang2009}. 

\noindent We also need to assume that
\begin{equation}\label{tech-assump}
\mu_F<\min\{\mu_M,\gamma+\mu_{A,1}\}.
\end{equation}
Assumption (\ref{tech-assump}) is also consistent with parameter values considered for \textit{Aedes spp.} in \cite{Anguelov2012TIS,Anguelov2019,Bliman2019,Strugarek2019}. 

\subsection{Existence and uniqueness of solutions for model (\ref{SIT-pde})} 
Here, we first address the question of
existence and uniqueness of \tcm{solutions} of the reaction-diffusion (RD)
system (\ref{SIT-pde}) in unbounded domains. For this purpose we
will use materials recalled in \ref{appendix-material}.

\par

Let $C_{ub}(\R)$ be the Banach space of bounded, uniformly
continuous function on $\R$ and, $$C^{2}_b(\R)=\{f\in C_{ub}(\R):
f'\in C_{ub}(\R),\quad  f''\in
C_{ub}(\R)\}.$$ $C_{ub}(\R)$ and $C^{2}_b(\R)$ are endowed wit the
following (sup) norms
\begin{equation}\label{borme-C-ub}
\|f\|_{C_{ub}(\R)}=\|f\|_{\infty}=\sup\limits_{x\in\R}|f(x)|
\end{equation}

and

\begin{equation}\label{norme-C-b}
\|g\|_{C_{b}^2(\R)}=\|g\|_{C_{ub}(\R)}+\|g'\|_{C_{ub}(\R)}+\|g''\|_{C_{ub}(\R)}.
\end{equation}

$C_{b}^2(\R)$ endowed with the norm $\|\cdot\|_{C_{b}^2(\R)}$ is a
Banach space. 

We set $w=(A,M,F)'$.  System (\ref{SIT-pde}) can be written as the abstract Cauchy problem
\begin{equation}\label{Abstract-form}
\left\{
 \begin{array}{l}
 \displaystyle\frac{d w}{d t} +A w= H(w),\\
 w(0)=w_0
 \end{array}
 \right.
\end{equation}
where in the Banach space $B=C_{ub}(\R)\times C_{ub}(\R)\times
C_{ub}(\R)$ we have, 
\begin{equation}\label{RD-setup}
\left\{
\begin{array}{l}
\eta=diag(0,d_M,d_F),\\
  D(A) = C^{2}_b(\R)\times C^{2}_b(\R)\times C^{2}_b(\R), \\
  Aw = -\eta w'',\\
  H: D(A) \rightarrow D(A),\\
   H(w)=\left(\phi F-(\gamma+\mu_{A,1}+\mu_{A,2}A)A,(1-r)\gamma A-\mu_MM, \displaystyle\frac{ M}{M+M_T}r\gamma A-\mu_FF\right)'.
\end{array}
\right.
\end{equation}
For $E\in\{C_{ub}(\R),C^{2}_b(\R)\}$ and $(a,b,c)\in E\times E\times E$, we define the norm
$$\|(a,b,c)\|_{E\times E\times E}=\|a\|_{E}+\|b\|_{E}+\|c\|_{E}.$$
From \cite[Theorem 2.1]{Rauch1978}, we
deduce the following Proposition \ref{Proposition-Global-existence}.

\begin{proposition}\label{Proposition-Global-existence}
For any $w_0\in B$ there is a positive constant $T > 0$ depending
only on $H$ and $\|w_0\|_{B}$, such that system (\ref{Abstract-form}) in
$[0, T]$, admits a unique local solution $w\in C([0, T], B)$ and
  \begin{equation}\label{integral-u}
    w(t)=S(t)w_0+\displaystyle\int_0^tS(t-\tau)H(w(\tau))d\tau, \quad
    \forall t\in[0, T]
\end{equation}
where $\{S(t)\}_{t\geq0}$ is the Gauss-Weierstrass $C_0-$semigroup
of contractions defined on the Banach space $B$ (see also
(\ref{Gauss-semigroup-n})-(\ref{Gauss-Rauch-Smoller}), page
\pageref{Gauss-semigroup-n}).
\end{proposition}

In order to prove global (in time) existence of solutions system (\ref{Abstract-form}), we use the notion of invariant regions, see e.g. \cite[Chapter 14, pages 198-212]{Smoller1983} and \cite{Rauch1978}.

\begin{lemma}\label{boundeness-lemma}(Invariant rectangles)\\ Let $k_1$ and $k_2$ be two real numbers such that $k_1>0$ and $k_2\geq1$. The following results hold true.
	\begin{enumerate}
		\item Assume that $\rr\leq1$. The set 
		$$\Gamma_{\rr\leq1}=\left\{(A,M,F): (0,0,0)'\leq(A,M,F)'\leq \left(k_1,\dfrac{(1-r)\gamma}{\mu_M}k_1,\dfrac{r\gamma}{\mu_F}k_1\right)' \right\}$$ is positively invariant for system (\ref{Abstract-form}).
		\item Assume that $\rr>1$. The set 
		$$\Gamma_{\rr>1}=\left\{(A,M,F)':(0,0,0)'\leq (A,M,F)'\leq \left(k_2A^*,k_2M^*,k_2F^*\right)'\right\}$$ is positively invariant for system (\ref{Abstract-form}).
	\end{enumerate}
	
\end{lemma}

\begin{proof}
	See \ref{AppendixA0}.
\end{proof}

From the local existence result and the existence of invariant rectangles, we deduce the following global existence result (see e.g. \cite[page 307]{Logan2008}, \cite{Rauch1978}).

\begin{proposition}\label{Proposition-Global-existencebis}
	For any $w_0\in B$, system (\ref{Abstract-form}) admits a unique global solution $w\in C([0, +\infty), B)$.
\end{proposition}

\subsection{Existence of traveling waves for model (\ref{SIT-pde})}

In compact form, model (\ref{SIT-pde}) can be rewritten as follows

\begin{equation}\label{compact-STI-pde}
\displaystyle\frac{\partial U}{\partial
t}=D\displaystyle\frac{\partial^2 U}{\partial x^2}+H_{M_T}(U), \quad
(t,x)\in\mathbb{R}_+\times\mathbb{R},
\end{equation}
where
\begin{equation}\label{definition-compact-STI-pde}
\begin{array}{rcl}
  U & = & (A, M, F)', \\
  D & = & diag(0, d_M, d_F), \\
  H_{M_T}(U) & = & \left(
               \begin{array}{c}
                \phi F-(\gamma+\mu_{A,1}+\mu_{A,2}A)A \\
                 %r\gamma A-(+\mu_Y)Y \\
                 (1-r)\gamma A-\mu_MM \\
                 \displaystyle\frac{ M}{M+M_T}r\gamma A-\mu_FF \\
               \end{array}
             \right).
\end{array}
\end{equation}
To study the traveling wave problem, we consider solution of
(\ref{definition-compact-STI-pde}) of the form
\begin{equation}\label{tw-setting}
    \begin{array}{rcl}
      A(t,x) & = & A(z), \quad \mbox{ with }z=x+ct\\
      %Y(t,x) & = & Y(z), \\
      M(t,x) & = & M(z), \\
      F(t,x) & = & F(z),
    \end{array}
\end{equation}
where $c$ is the wave speed. Therefore, \tcm{a} traveling wave solution of
(\ref{compact-STI-pde}) satisfies

\begin{equation}\label{tw-compact-STI-pde}
DU''-cU'+H_{M_T}(U)=0,
\end{equation}
where $':=\frac{d}{dz}$. We will further assume that
$$U(-\infty)=E_{-\infty}, \qquad U(+\infty)=E_{+\infty}$$
where $E_{-\infty}$ and $E_{+\infty}\in\{E^*;E_1;E_2\}$ are two distinct homogeneous
equilibria of (\ref{compact-STI-pde}) with $E_{-\infty}<E_{+\infty}$.

\subsubsection{Existence of monostable traveling waves for system
(\ref{SIT-pde})}

In this section, we first recall some useful results \cite{Fang2009}. Consider the $n$-dimensional ($n\geq2$) reaction-diffusion system
\begin{equation}\label{RD-Fang}
\displaystyle\frac{\partial u_i}{\partial
t}=d_i\displaystyle\frac{\partial^2 u_i}{\partial
x^2}+f_i(u_1,...,u_n), \quad
(t,x)\in\mathbb{R}_+\times\mathbb{R},\quad 1\leq i\leq n,
\end{equation}
where, some, but not all, diffusion coefficients $d_i$ are zero, and
the others are positive. \tcm{Let us set $D:=diag(d_1,...,d_n)$.}

\noindent Recall that for a square matrix $Y$, the stability modulus is defined as follows: 
$$s(Y)=\max\{\text{Re}\sigma / \det(\sigma I-Y)=0\}.$$ 
To prove the existence of monostable
traveling wave solutions for system (\ref{RD-Fang}), Fang and Zhao \cite{Fang2009}
consider the following assumptions:

\vspace{0.25cm}
\noindent \emph{(\textbf{H})} Assume that
$f=(f_1,...,f_n)':\mathbb{R}^n\rightarrow \mathbb{R}^n$ satisfies the
following assumptions:
\begin{enumerate}
    \item $f$ is continuous with $f(\textbf{0})=f(\textbf{1})=\textbf{0}$ and there is no $\nu$ other than \textbf{0} and \textbf{1} such
that $f(\nu) =$ \textbf{0} with \textbf{0} $\leq\nu\leq$ \textbf{1}.
    \item System (\ref{RD-Fang}) is cooperative.
    \item $f(u)$ is piecewise continuously differentiable in $u$ for
    \textbf{0} $\leq\nu\leq$ \textbf{1} and differentiable at \textbf{0}, and the matrix $f' (\textbf{0})$ is irreducible, 
    with $s(f'(\textbf{0}))>0$.
\end{enumerate}

For practical applications, $f'(\bf{0})$ irreducible, in item \emph{(\textbf{H})}$_{3}$, is quite
restrictive. However in the proof of their results, Fang and Zhao \cite{Fang2009} needed only a consequence of this 
irreducibility property, deduced also from the Perron-\tcm{Frobenius} Theorem (see \cite[chapter 4, section 3]{Smith1995}). So we can weaken this irreducibility assumption and consider its consequence, such that \emph{(\textbf{H})} now becomes:

\vspace{0.25cm}
\noindent \emph{(\textbf{H'})} Assume that
$f=(f_1,...,f_n)':\mathbb{R}^n\rightarrow \mathbb{R}^n$ satisfies assumptions (\textbf{H})$_1$, (\textbf{H})$_2$, and 
\begin{enumerate}
%    \item $f$ is continuous with $f(\textbf{0})=f(\textbf{1})=\textbf{0}$ and there is no $\nu$ other than \textbf{0} and \textbf{1} such
%that $f(\nu) =$ \textbf{0} with \textbf{0} $\leq\nu\leq$ \textbf{1}.
%    \item System (\ref{RD-Fang}) is cooperative.
    \item [(\textbf{H}')$_3$] $f(u)$ is piecewise continuously differentiable in $u$ for
    \textbf{0} $\leq\nu\leq$ \textbf{1} and differentiable at \textbf{0}, and for $\mu>0$ the matrix $A(\mu):=\mu^2D+f' (\emph{\textbf{0}})$ is
    such that $\lambda(\mu)=s(A(\mu))>0$ is a simple eigenvalue
of $A(\mu)$ with a strongly positive eigenvector
$v(\mu)=(v_1(\mu),...,v_n(\mu))$ with $\|v(\mu)\|=1$.
\end{enumerate}

For $\mu>0$, we define the function $\Phi(\mu):=\lambda(\mu)/\mu>0$ and
$\overline{c}:=\inf\limits_{\mu>0}\Phi(\mu)>0.$ Suppose also that
$\overline{\mu}\in(0,+\infty)$ is the value of $\mu$ at which
$\Phi(\mu)$ attains its infimum. The following result is valid

\begin{lemma}\label{lemme-Fang-Zhao}
Assume that \emph{(\textbf{H'})} holds. Let $\phi\in
\{u\in\mathcal{C}:\textbf{0}\leq u\leq \textbf{1}\}$ and
$u(t,x;\phi)$ be the unique solution of (the integral form of)
(\ref{RD-Fang}) through $\phi$. Then there exists a real number
$c^*\geq\overline{c}>0$ such that the following statements are
valid:
\begin{itemize}
    \item[(i)] If $\phi$ has compact support, then $\lim\limits_{t\rightarrow+\infty, |x|\geq
    ct}u(t,x;\phi)=\textbf{0}$, $\forall c>c^*$.
    \item[(ii)] For any $c \in (0, c^*)$ and $r > 0$, there is a positive number $R_r$ such that for any $\phi\in
\{u\in\mathcal{C}:\textbf{0}\leq u\leq \textbf{1}\}$ with $\phi\geq
\textbf{r}$ on an interval of length $2R_r$ , there holds
$\lim\limits_{t\rightarrow+\infty, |x|\leq
    ct}u(t,x;\phi)=\textbf{1}$.
    \item[(iii)] If, in addition, $f(\min\{\rho v(\overline{\mu}),\textbf{\emph{1}}\})\leq\rho
    f'(\textbf{\emph{0}})v(\overline{\mu})$, $\forall\rho>0$, then
    $c^*=\overline{c}$.
\end{itemize}
\end{lemma}

\begin{proof}
The results follow from 
\cite[Lemma 2.3]{Fang2009} by considering assumption \emph{(\textbf{H'})} instead of \emph{(\textbf{H})}.
\end{proof}

We observe that in Lemma \ref{lemme-Fang-Zhao}, assumption stated in
item $(iii)$ implies that $f(u)$ is dominated by its linearization
at \textbf{0} in the direction of $v(\overline{\mu})$ and this
ensures the so-called linear determinacy property (\cite{Weinberger2002}, \cite{Li2005} and
references therein).
\par

Fang and Zhao \cite{Fang2009} further consider the following
assumption.\\
 \emph{(\textbf{K})} Assume that
$f=(f_1,...,f_n)':\mathbb{R}^n\rightarrow \mathbb{R}^n$ \tcm{is such that}:
\begin{enumerate}
    \item $f$ is continuous with $f(\textbf{0})=f(\textbf{1})=\textbf{0}$ and there is no $\nu$ other than \textbf{0} and \textbf{1} such
that $f(\nu) =$ \textbf{0} with \textbf{0} $\leq\nu\leq$ \textbf{1}.
    \item System (\ref{RD-Fang}) is cooperative.
    \item $f(u)$ is piecewise continuously differentiable in $u$ for
    \textbf{0} $\leq\nu\leq$ \textbf{1} and differentiable at \textbf{0}, and the matrix $f'(\textbf{0})$ is irreducible
    with $s(f'(\textbf{0}))>0$.
    \item There exist $a>0$, $\sigma>1$ and $r>0$ such that $f(u)\geq
    f'(\textbf{0})u-a\|u\|^\sigma$\textbf{1} for all
    \textbf{0}$\leq u\leq\textbf{\emph{r}}$.
    \item For any $\rho>0$, $f(\min\{\rho v(\mu),\textbf{1}\})\leq\rho
    f'(\textbf{0})v(\mu)$, $\forall \mu\in(0,\overline{\mu}]$, where $\overline{\mu}$ is the
value of $\mu$ at which $\Phi(\mu)$ attains its infimum.
\end{enumerate}

As previously, we weaken the third statement of \emph{(\textbf{K})}
and it now reads as:\\

\emph{(\textbf{K'})} Assume that
$f=(f_1,...,f_n):\mathbb{R}^n\rightarrow \mathbb{R}^n$ satisfies assumptions (\textbf{K})$_1$, (\textbf{K})$_2$, (\textbf{K})$_4$, (\textbf{K})$_5$, and 
\begin{enumerate}
%    \item $f$ is continuous with $f(\textbf{0})=f(\textbf{1})=\textbf{0}$ and there is no $\nu$ other than \textbf{0} and \textbf{1} such
%that $f(\nu) =$ \textbf{0} with \textbf{0} $\leq\nu\leq$ \textbf{1}.
%    \item System (\ref{RD-Fang}) is cooperative.
    \item [(\textbf{K}')$_3$] $f(u)$ is piecewise continuously differentiable in $u$ for
    \textbf{0} $\leq\nu\leq$ \textbf{1} and differentiable at \textbf{0}, and for $\mu>0$ the matrix $A(\mu):=\mu^2D+f' (\textbf{0})$ is
    such that $\lambda(\mu)=s(A(\mu))>0$ is a simple eigenvalue
of $A(\mu)$ with a strongly positive eigenvector
$v(\mu)=(v_1(\mu),...,v_n(\mu))$ with $\|v(\mu)\|=1$.
%    \item There exist $a>0$, $\sigma>1$ and $r>0$ such that $f(u)\geq
%    f'(\textbf{0})u-a\|u\|^\sigma$\textbf{1} for all
%    \textbf{0}$\leq u\leq\textbf{\emph{r}}$.
%    \item For any $\rho>0$, $f(\min\{\rho v(\mu),\textbf{1}\})\leq\rho
%    f'(\textbf{0})v(\mu)$, $\forall \mu\in(0,\overline{\mu}]$, where $\overline{\mu}$ is the
%value of $\mu$ at which $\Phi(\mu)$ attains its infimum.
\end{enumerate}
Then, the following results hold
%\begin{lemma}\label{lemme-Fang-Zhao2}
%If there is a ball $B$ in $\mathbb{R}^n$ centered at
%\textbf{\emph{0}} such that $\displaystyle\frac{\partial^2f_k(u)}{\partial
%u_i\partial u_j}$, $\forall 1\leq i,j,k\leq n$, is continuous at
%every point of $B$, then \emph{(\textbf{K'})}$_4$ holds.
%\end{lemma}
%
%\begin{proof}
%The results follow from 
%\cite[Lemma 3.1]{Fang2009} by considering assumption \emph{(\textbf{K'})} instead of \emph{(\textbf{K})}.
%\end{proof}

\begin{theorem}\label{theorem-Fang-Zhao-monostable}
Assume that \emph{(\textbf{K'})} holds and let $c^*$ be defined as
in Lemma \ref{lemme-Fang-Zhao}. Then for each $c\geq c^*$, system
(\ref{RD-Fang}) has a nondecreasing wavefront $U(x + ct)$ connecting
\textbf{\emph{0}} and
\textbf{\emph{1}}; while for any $c \in (0, c^*)$, there is
no wavefront $U(x + ct)$ connecting \textbf{\emph{0}} and
\textbf{\emph{1}}.
\end{theorem}
\begin{proof}
The results follow from 
\cite[Theorem 3.1]{Fang2009} by considering (\textbf{K'}) instead of (\textbf{K}).
\end{proof}

Now we are in position to study the existence of monostable traveling wave solutions for model
(\ref{compact-STI-pde}) when no SIT control occurs, i.e. $M_T=0$. System (\ref{compact-STI-pde}) becomes
\begin{equation}\label{H-0}
   H_{M_T}(U)=H_{0}(U)=\left(
               \begin{array}{c}
                \phi F-(\gamma+\mu_{A,1}+\mu_{A,2}A)A \\
                 %r\gamma A-(+\mu_Y)Y \\
                 (1-r)\gamma A-\mu_MM \\
                 r\gamma A-\mu_FF \\
               \end{array}
             \right).
\end{equation}

From section \ref{preliminay}, we deduce that $H_0(U)=\textbf{0}$ has two solutions
$E_0=\textbf{0}$ and $E^*=(A^*,M^*,F^*)^T$, with $s(H'_0(E_0))>0$ and $s(H'_0(E^*))<0$ where $H'_0(U)$ denotes the
Jacobian matrix of $H_0$ at $U$. To follow the ideas developed by
Fang and Zhao \cite{Fang2009} and for sake of clarity, we
first normalize system (\ref{compact-STI-pde}) -(\ref{H-0}). For
this purpose, we set
$$a=A/A^*,\quad m=M/M^*,\quad f=F/F^*.$$ Thus, system (\ref{compact-STI-pde})
-(\ref{H-0}) becomes
\begin{equation}\label{compact-STI-pde-normalise}
\displaystyle\frac{\partial u}{\partial
t}=D\displaystyle\frac{\partial^2 u}{\partial x^2}+h_0(u), \quad
(t,x)\in\mathbb{R}_+\times\mathbb{R},
\end{equation}
where
\begin{equation}\label{definition-compact-STI-pde}
\begin{array}{lcl}
  u & = & (a, m, f)^T, \\
  D & = & diag(0, d_M, d_F), \\
  h_0(u) & = & \left(
               \begin{array}{c}
                (\gamma+\mu_{A,1})(\rr f-a-(\rr-1)a^2) \\
                 \mu_M(a-m) \\
                 \mu_F(a-f) \\
               \end{array}
             \right).
\end{array}
\end{equation}
Thus, obviously, system (\ref{compact-STI-pde-normalise}) is cooperative and admits only two homogeneous
equilibria, $e_0=\bf{0}$ and $e^*=\bf{1}$, such that (\textbf{H'})$_1$ and (\textbf{H'})$_2$ are verified. Similarly, since the right hand-side is a two-order polynomial, it is easy to show that (\textbf{K'})$_4$ holds too, i.e. $h_0(u)\geq h'_0(\textbf{0})u-(\rr-1)(\mu_{A,1}+\gamma)\|u\|^2 \textbf{1}$, for all $\textbf{0}\leq u \leq \textbf{1}$. 

\noindent Now, we need to check (\textbf{K'})$_3$ and (\textbf{K'})$_5$. For $\mu>0$ let us define
\begin{equation}\label{omega-mu}
\omega(\mu)=\mu^2D+h_0'(\textbf{0})=\left(
\begin{array}{ccc}
-(\mu_{A,1}+\gamma) & 0 & \rr(\mu_{A,1}+\gamma) \\
\mu_M & -\mu_M+\mu^2d_M & 0 \\
\mu_F & 0 & -\mu_F+\mu^2d_F \\
\end{array}
\right).
\end{equation}
The eigenvalues of $\omega(\mu)$ are $\sigma_3=-\mu_M+\mu^2d_M$ and the 
solutions of the second order equation
\begin{equation}
\label{rootsigma}
\sigma^2+\sigma(-\mu^2d_F+\mu_F+\mu_{A,1}+\gamma)+\mu_F(\mu_{A,1}+\gamma)(1-\rr)-(\mu_{A,1}+\gamma)\mu^2d_F=0.
\end{equation}
Since 
$\Delta=(-\mu^2d_F+\mu_F+\mu_{A,1}+\gamma)^2+4\mu_F(\mu_{A,1}+\gamma)(\rr-1)+4(\mu_{A,1}+\gamma)\mu^2d_F>0$, we deduce that
$$\sigma_1=\displaystyle\frac{-(-\mu^2d_F+\mu_F+\mu_{A,1}+\gamma)-\sqrt{\Delta}}{2}<0,$$

$$\sigma_2=\displaystyle\frac{-(-\mu^2d_F+\mu_F+\mu_{A,1}+\gamma)+\sqrt{\Delta}}{2}>0.$$

Obviously, we have $\sigma_2>\sigma_1$. Since $\mu_F<\mu_M$ and $d_F\geq d_M$, we have
$$
\sigma_2-\sigma_3=\frac{\sqrt{\Delta}-(\mu_{A,1}+\gamma+\mu_F-\mu^2d_F)}{2} + \left(\mu_M - \mu^2d_M\right) \geq \frac{\sqrt{\Delta}-(\mu_{A,1}+\gamma-(\mu_F-\mu^2d_F))}{2},
$$
that is 
$$
\begin{array}{lcl}
\sigma_2-\sigma_3 \geq \dfrac{\Delta-(\mu_{A,1}+\gamma-(\mu_F-\mu^2d_F))^2}{2(\sqrt{\Delta}+(\mu_{A,1}+\gamma+\mu_F-\mu^2d_F)}&=&\dfrac{4\mu_F(\mu_{A,1}+\gamma)\rr}{2(\sqrt{\Delta}+(\mu_{A,1}+\gamma+\mu_F-\mu^2d_F)}\\
&=&\tcm{\dfrac{\mu_F(\mu_{A,1}+\gamma)\rr}{-\sigma_1}}>0,
\end{array}
$$
such that $\sigma_2>\max(\sigma_1,\sigma_3)$. 

Now we compute an eigenvector
$v_0=(x,y,z)'$ of $\omega(\mu)$ associated to $\sigma_2$. We need to
solve the algebraic equations
$$
\left\{
\begin{array}{rcc}
  -(\mu_{A,1}+\gamma)x +R(\mu_{A,1}+\gamma)z& = & \sigma_2x, \\
  \mu_Mx-\mu_My+\mu^2d_My & = & \sigma_2y, \\
  \mu_Fx+(\mu^2d_F-\mu_F)z & = & \sigma_2z.
\end{array}
\right.
$$
From the last two equations, since $\sigma_2>\mu^2d_F-\mu_F$ and $\sigma_2>\sigma_3$, we obtain
$$
\left\{
\begin{array}{ccl}
  y & = & \dfrac{\mu_M}{\sigma_2+\mu_M-\mu^2d_M}x=\dfrac{\mu_M}{\sigma_2-\sigma_3}x, \\
  z & = & \displaystyle\frac{\mu_F}{\sigma_2+\mu_F-\mu^2d_F}x.
\end{array}
\right.
$$
Substituting into the first equation leads to
$$
\begin{array}{cl}
   &  -(\mu_{A,1}+\gamma)x+\rr(\gamma+\mu)\displaystyle\frac{\mu_F}{\sigma_2+\mu_F-\mu^2d_F}x=\sigma_2x, \\
\Leftrightarrow   & (-(\mu_{A,1}+\gamma)(\sigma_2+\mu_F-\mu^2d_F+\rr(\mu_{A,1}+\gamma)\mu_F))x=(\sigma_2+\mu_F-\mu^2d_F)\sigma_2x \\
\Leftrightarrow & (\sigma_2^2+\sigma_2(-\mu^2d_F+\mu_F+\mu_{A,1}+\gamma)+\mu_F(\mu_{A,1}+\gamma)(1-\rr)-(\mu_{A,1}+\gamma)\mu^2d_F)x=0.
%\Rightarrow & x\in\mathbb{R} ~ \mbox{since} ~ \sigma_2^2+\sigma_2(-\mu^2d_F+\mu_F+\mu_{A,1}+\gamma)+\mu_F(\mu_{A,1}+\gamma)(1-\rr)-(\mu_{A,1}+\gamma)\mu^2d_F=0.\\
\end{array}
$$
Since $\sigma_2$ is a positive root of \eqref{rootsigma}, we deduce that $x \in \mathbb{R}^*$.
%In addition, one also has . Indeed,
%$$
%\begin{array}{cl}
%   & 4\mu_F(\mu_{A,1}+\gamma)(\rr-1)+4(\gamma+\mu_{A,1})\mu_F>0 \\
%  \Leftrightarrow & 4\mu_F(\mu_{A,1}+\gamma)(\rr-1)+4(\gamma+\mu_{A,1})\mu_F+4(\mu_{A,1}+\gamma)\mu^2d_F-4(\mu_{A,1}+\gamma)\mu^2d_F>0 \\
%  \Leftrightarrow & (\mu^2d_F-\mu_F-\mu_{A,1}-\gamma)^2+4(\mu_{A,1}+\gamma)(\mu_F(\rr-1)+\mu^2d_F)>(\mu^2d_F-\mu_F+\mu_{A,1}+\gamma)^2 \\
%  \Leftrightarrow & \sqrt{\Delta}> (\mu^2d_F-\mu_F+\mu_{A,1}+\gamma)>0 \quad\mbox{beacause of}\quad (\ref{tech-assump})\\
%  \Leftrightarrow & \sqrt{\Delta}>
%  2(\mu^2d_F-\mu_F)-(\mu^2d_F-\mu_F)+\mu_{A,1}+\gamma\\
%\Leftrightarrow & \mu^2d_F-\mu_F-\mu_{A,1}-\gamma+\sqrt{\Delta}>
%  2(\mu^2d_F-\mu_F)\\
%\Leftrightarrow &  \sigma_2>\mu^2d_F-\mu_F.
%\end{array}
%$$
Therefore, setting $x=1$, we deduce
$$v_0=\left(1,\displaystyle\frac{\mu_M}{\sigma_2+\mu_M-\mu^2d_M},\displaystyle\frac{\mu_F}{\sigma_2+\mu_F-\mu^2d_F}\right)'>\textbf{0}.
$$
Then, we set
\begin{equation}\label{vecteur-propre-monostable}
v(\mu)=\displaystyle\frac{1}{\|v_0\|}v_0
\end{equation}
 such that $\|v(\mu)\|=1$. Thus (\textbf{K'})$_3$ holds.
 
\noindent Let us consider, for $\mu>0$, $$\Phi(\mu)=\sigma_2(\mu)/\mu>0.$$
Since $\Phi(\mu)>\mu d_F-\displaystyle\frac{\mu_F}{\mu}$, it implies
that $\lim\limits_{\mu\rightarrow+\infty}\Phi(\mu)=+\infty$.
Similarly, since $\lim\limits_{\mu\rightarrow0}\sigma_2(\mu)>0$, we
also deduce that $\lim\limits_{\mu\rightarrow0^+}\Phi(\mu)=+\infty$.
In addition, $$\Phi'(\mu)=0\Leftrightarrow
\psi(\mu):=\mu\sigma'_2(\mu)-\sigma_2(\mu)=0.$$ Direct computations
lead that
$$\psi(\mu)=\displaystyle\frac{\mu^2d_F+\mu_F+\mu_{A,1}+\gamma}{2}+\displaystyle\frac{\mu^2d_F(\mu^2d_F-\mu_F+\mu_{A,1}+\gamma)}{\sqrt{\Delta}}-\displaystyle\frac{\Delta}{2\sqrt{\Delta}}.$$

Thus, $\psi(\mu)=0$ is equivalent to
$$(\mu^2d_F+\mu_F+\mu_{A,1}+\gamma)^2\Delta-(2\mu^2d_F(\mu^2d_F-\mu_F+\mu_{A,1}+\gamma)-\Delta)^2=0.$$
Let us set:
$$
\begin{array}{cl}
  x & =\mu^2d_F, \\
  a_1 &= \mu_{A,1}+\gamma, \\
  a_2 & =\mu_F+a_1, \\
  a_3 &=\mu_F(\rr-1),
\end{array}
$$
such that $\Delta=(a_2-x)^2+4a_1(a_3+x)$.
Then, solving $\psi(\mu)=0$ is equivalent to
$$(x+a_2)^2((a_2-x)^2+4a_1(x+a_3))-(2x(x-a_2+2a_1)-((a_2-x)^2+4a_1(x+a_3)))^2=0,$$
$$((x+a_2)(a_2-x))^2-(2x(x+2a_1)-a_2^2-x^2+4a_1(x+a_3)))^2+4a_1(x+a_2)^2(x+a_3)=0,$$
$$((x+a_2)(a_2-x))^2-(x^2-a_2^2-4a_1a_3))^2+4a_1(x+a_2)^2(x+a_3)=0,$$
$$(a_2^2-x^2)^2 -(x^2-a_2^2-4a_1a_3))^2+4a_1(x+a_2)^2(x+a_3)=0,$$
$$-4a_1a_3(2a_2^2-2x^2+4a_1a_3)+4a_1(x+a_2)^2(x+a_3)=0.$$
Simplifying by $4a_1$ and expanding the previous terms lead to
%$$-a_3(2a_2^2-2x^2+4a_1a_3)+(x+a_2)^2(x+a_3)=0,$$
%$$x^3+a_3x^2+2a_2x^2+2a_2a_3x+a_2^2x-2a_32a_2^2+2a_3x^2-4a_1a_3^2=0,$$
\begin{equation}
x^3+(3a_3+2a_2)x^2+(2a_2a_3+a_2^2)x-(a_3a_2^2+4a_1a_3^2)=0.
\label{poly-vitesse}
\end{equation}
All coefficients are positive except the last one, such that $\psi(\mu)=0$ has a unique positive root which ensures that $\Phi'$ changes sign once on $(0, +\infty)$.
Thus, taking into account the computations on the limits of $\Phi$ and the
continuity of $\Phi$, we deduce that there exists a unique
$\overline{\mu}>0$ such that
$\Phi(\overline{\mu})=\inf\limits_{\mu>0}\Phi(\mu)>0$, the so-called minimal speed
\begin{equation}\label{c-bar}
    \overline{c}:=\dfrac{\sigma_2(\overline{\mu})}{\overline{\mu}}.
\end{equation}

Let $\mu\in(0,\bar{\mu}]$ and $\rho>0$.

\begin{equation}\label{inequality-mu}
\begin{array}{cll}
h_0(\rho v(\mu)) &=&\left(
                     \begin{array}{c}
                       (\mu_{A,1}+\gamma)(\rr\rho v_3-\rho v_1-(\rr-1)(\rho v_1)^2) \\
                       \rho\mu_M(v_1-v_2) \\
                       \rho\mu_F(v_1-v_3) \\
                     \end{array}
                   \right)\\
                   &\leq& \rho\left(
                     \begin{array}{c}
                       (\mu_{A,1}+\gamma)(\rr v_3- v_1) \\
                       \mu_M(v_1-v_2) \\
                       \mu_F(v_1-v_3) \\
                     \end{array}
                   \right)\\
                   &=&\rho h_0'(\textbf{0})v(\mu)
\end{array}
\end{equation}

where $v(\mu)=(v_1(\mu),v_2(\mu),v_3(\mu))'$ is defined in
(\ref{vecteur-propre-monostable}). Therefore, since (\textbf{H}') and (\textbf{K}') hold true, we
can apply Lemma \ref{lemme-Fang-Zhao}, page \pageref{lemme-Fang-Zhao}, and Theorem \ref{theorem-Fang-Zhao-monostable}, page \pageref{theorem-Fang-Zhao-monostable} to system (\ref{compact-STI-pde-normalise}), to deduce that the spreading speed $c^*$ coincides with the minimal wave speed,  $\overline{c}$, defined in (\ref{c-bar}), and the following result for system (\ref{SIT-pde})

%\begin{lemma}\label{lemme-PDE-monostable1}
% Let $\phi\in\{u\in\mathcal{C}:\textbf{\emph{0}}\leq u\leq \textbf{\emph{1}}\}$ and
%$u(t,x;\phi)$ be the unique solution of (the integral form of)
%(\ref{compact-STI-pde-normalise}) through $\phi$. Then there exists
%a real number $c^*\geq\overline{c}>0$, where $\overline{c}$ is
%defined in (\ref{c-bar}) and is  such that the following statements
%are valid:
%\begin{itemize}
%    \item[(i)] If $\phi$ has compact support, then $\lim\limits_{t\rightarrow+\infty, |x|\geq
%    ct}u(t,x;\phi)=\textbf{\emph{0}}$, $\forall c>c^*$.
%    \item[(ii)] For any $c \in (0, c^*)$ and $r > 0$, there is a positive number $R_r$ such that for any $\phi\in
%\{u\in\mathcal{C}:\textbf{\emph{0}}\leq u\leq \textbf{\emph{1}}\}$
%with $\phi\geq \textbf{r}$ on an interval of length $2R_r$ , there
%holds $\lim\limits_{t\rightarrow+\infty, |x|\leq
%    ct}u(t,x;\phi)=\textbf{\emph{1}}$.
%\end{itemize}
%\end{lemma}
%
%\begin{proof}
%Since \emph{(\textbf{H'})} holds, the conclusions follow from Lemma
%\ref{lemme-Fang-Zhao}.
%\end{proof}

\begin{theorem}\label{theorem-PDE-monostable}
For each $c\geq c^*=\overline{c}$, system
(\ref{SIT-pde}) has a nondecreasing wavefront $U(x
+ ct)$ connecting $E_0$ and $E^*$; while for
any $c \in (0, c^*)$, there is no wavefront $U(x + ct)$ connecting
$E_0$ and $E^*$.
\end{theorem}

%\begin{proof}
%Assumptions of Theorem \ref{theorem-PDE-monostable} together with
%inequality (\ref{inequality-mu}) imply that the requirement of item
%(iii) of Lemma \ref{lemme-Fang-Zhao} holds. Therefore,
%$c^*=\overline{c}$. Once more, assumptions of Theorem
%\ref{theorem-PDE-monostable} together with inequality
%(\ref{inequality-mu}) and Lemma \ref{lemme-Fang-Zhao2} imply that
% \emph{(\textbf{K'})} holds. Therefore, the conclusion follows from
% Theorem \ref{theorem-Fang-Zhao-monostable}.
%\end{proof}

\subsubsection{Existence of bistable traveling waves for system (\ref{SIT-pde})}

To show existence of bistable traveling wave solutions of system
(\ref{RD-Fang}), Fang and Zhao 
\cite{Fang2009} considered the following assumptions:\\
\emph{(\textbf{L})} Assume that
$f=(f_1,...,f_n)'\in\mathcal{C}^1(\mathbb{R}^n, \mathbb{R}^n)$
satisfies the following assumptions:
\begin{enumerate}
    \item $f(\textbf{0})=f(\textbf{1})=f(\alpha)=\textbf{0}$ with
    \textbf{0}$\ll\alpha\ll$\textbf{1}. There is no $\nu$ other than \textbf{0}, \textbf{1} and $\alpha$ such
that $f(\nu) =$ \textbf{0} with \textbf{0} $\leq\nu\leq$ \textbf{1}.
    \item System (\ref{RD-Fang}) is cooperative.
    \item $u\equiv\textbf{0}$ and $u\equiv\textbf{1}$ are stable,
    and $u\equiv\alpha$ is unstable, that is,
    $$\lambda_0:=s(f'(\textbf{0}))<0, \quad \lambda_1:=s(f'(\textbf{1}))<0, \quad \lambda_\alpha:=s(f'(\alpha))>0.$$
    \item $f'(\textbf{0})$, $f'(\textbf{1})$ and $f'(\alpha)$ are
    irreducible.
\end{enumerate}
Then, they showed the following result

%The main result in \cite{Fang2009} about bistable traveling wave solutions for system (\ref{RD-Fang}) is as follows

\begin{theorem}\label{Theorem-Fang-Zhao}(
\cite[Theorem 4.1]{Fang2009} - bistable TW)\\Assume that (\textbf{L}) holds. Then
system (\ref{RD-Fang}) admits a monotone wavefront $(U,c)$ with
$U(-\infty)=\textbf{0}$ and $U(+\infty)=\textbf{1}.$
\end{theorem}

Practically, in assumption \emph{(\textbf{L})}, the last item concerning the irreducibility of matrices $f'(\textbf{0})$,
$f'(\textbf{1})$ and $f'(\alpha)$, is quite restrictive for application. We found that we can weaker this assumption as follows:\\

\noindent \emph{(\textbf{L'})}: assume that $f\in\mathcal{C}^1(\mathbb{R}^n,
\mathbb{R}^n)$ satisfies assumptions (\textbf{L})$_1$, (\textbf{L})$_2$, (\textbf{L})$_3$ and 
\begin{enumerate}
%    \item $f(\textbf{0})=f(\textbf{1})=f(\alpha)=\textbf{0}$ with
%    \textbf{0}$\ll\alpha\ll$\textbf{1}. There is no $\nu$ other than \textbf{0}, \textbf{1} and $\alpha$ such
%that $f(\nu) =$ \textbf{0} with \textbf{0} $\leq\nu\leq$ \textbf{1}.
%    \item System (\ref{RD-Fang}) is cooperative.
%    \item $u\equiv\textbf{0}$ and $u\equiv\textbf{1}$ are stable,
%    and $u\equiv\alpha$ is unstable, that is,
%    $$\lambda_0:=s(f'(\textbf{0}))<0, \quad \lambda_1:=s(f'(\textbf{1}))<0, \quad \lambda_\alpha:=s(f'(\alpha))>0.$$
    \item[(\textbf{L}')$_4$: ] There exists an eigenvector $e_0\gg\textbf{0}$ with $||e_0||_{\mathbb{R}^n}=1$ corresponding to
    $\lambda_0$, $f'(\textbf{1})$ and $f'(\alpha)$ are
    irreducible.
\end{enumerate}
\noindent or

\noindent \emph{(\textbf{L''})}: assume that $f\in\mathcal{C}^1(\mathbb{R}^n,\mathbb{R}^n)$ satisfies assumptions (\textbf{L})$_1$, (\textbf{L})$_2$, (\textbf{L})$_3$ and 
\begin{enumerate}
    \item[(\textbf{L}'')$_4$:] There exist eigenvectors $e_0\gg\textbf{0}$ and $e_1\gg\textbf{0}$ with $||e_0||_{\mathbb{R}^n}=||e_1||_{\mathbb{R}^n}=1$ corresponding to
    $\lambda_0$ and $\lambda_1$ respectively,  and $f'(\alpha)$ is
    irreducible.
\end{enumerate}
We obtain the following result
\begin{theorem}\label{Theorem-Fang-Zhao-new}Assume (\textbf{L'}) or (\textbf{L''}) holds. Then system
(\ref{RD-Fang}) admits a monotone wavefront $(U,c)$ with
$U(-\infty)=\textbf{0}$ and $U(+\infty)=\textbf{1}.$
\end{theorem}

\begin{proof}The proof is exactly the same as the proof of Theorem
4.1 in Fang and Zhao \cite{Fang2009}.
\end{proof}

Coming back to the SIT PDE model (\ref{compact-STI-pde}), with $0<M_T<M_{T_1}$, the
Jacobian matrix at the extinction equilibrium $E_0=\textbf{0}$ is
$$H'(E_0)=\left(
           \begin{array}{ccc}
             -(\gamma+\mu_{A,1}) & 0 & \phi \\
             (1-r)\gamma & -\mu_M & 0 \\
             0 & 0 & -\mu_F \\
           \end{array}
         \right).
$$
From (\ref{tech-assump}), we infer
$$\lambda_0=s(H'(E_0))=\max\{-(\gamma+\mu_{A,1}),-\mu_M,-\mu_F\}=-\mu_F.$$
Let $u=(x,y,z)'$ be an eigenvector of $H'(E_0)$ that correspond to
$\lambda_0$, that is

$$
\left\{\begin{array}{rcl}
  -(\gamma+\mu_{A,1})x+\phi z & = & \lambda_0x, \\
  (1-r)\gamma x -\mu_My & = & \lambda_0y, \\
  z & \in & \mathbb{R}.
\end{array}\right.
$$
From assumption (\ref{tech-assump}), that
$x=\displaystyle\frac{\phi}{\gamma+\mu_{A,1}-\mu_F}z$ and
$y=\displaystyle\frac{\phi(1-r)\gamma}{(\gamma+\mu_{A,1}-\mu_F)(\mu_M-\mu_F)}z$.
Therefore one can choose
$$u=(x,y,z)'=\left(\displaystyle\frac{\phi}{\gamma+\mu_{A,1}-\mu_F},\displaystyle\frac{\phi(1-r)\gamma}{(\gamma+\mu_{A,1}-\mu_F)(\mu_M-\mu_F)},1\right)'\gg\textbf{0}.$$
To be in line with the last point of assumption
\emph{(\textbf{L'})}, we set
$e_0=\displaystyle\frac{1}{||u||_{\mathbb{R}^3}}u\gg\textbf{0}$ so
that
$||e_0||_{\mathbb{R}^3}=1$. \\
In addition, let $(A,M,F)'$ be an homogeneous equilibrium of
(\ref{compact-STI-pde}); that is, $H(A,M,F)=0_{\mathbb{R}^3}$. The
Jacobian matrix at $(A,M,F)'$ is

$$H'(A,M,F)=\left(
           \begin{array}{ccc}
             -(\gamma+\mu_{A,1})-2\mu_{A,2}A & 0 & \phi \\
             (1-r)\gamma & -\mu_M & 0 \\
             \gamma r\displaystyle\frac{M}{M+M_T}  & r\gamma A\displaystyle\frac{M_T}{(M+M_T)^2} & -\mu_F \\
           \end{array}
         \right).
$$
Since $A>0$, $F>0$ and $M>0$ one deduces that $H'(A,M,F)$ is an irreducible matrix. Thus, (\textbf{L'})$_4$ holds true. 

%Here, in view of Theorem \ref{SIT-ode-theorem}, page \pageref{SIT-ode-theorem} we
%also consider the SIT-PDE model (\ref{SIT-pde}) when $$\rr>1 \quad
%\mbox{and} \quad 0<M_T<M_{T_1}$$ with
%$$\rr= \displaystyle\frac{r\gamma\phi}{(\gamma+\mu_{A,1})\mu_F},\quad
% M_{T_1}=\displaystyle\frac{(\sqrt{\rr}-1)^2}{Q}\quad
%\mbox{and}\quad Q=
%\displaystyle\frac{\mu_{A,2}\mu_M}{(\gamma+\mu_{A,1})(1-r)\gamma}.$$
In addition, since $0<M_T<M_{T_1}$, from Theorem \ref{SIT-ode-theorem}, page
\pageref{SIT-ode-theorem}, one deduces that the first and
third requirement of assumption \emph{(\textbf{L'})} are fulfilled
with $\alpha\equiv E_1$ and \textbf{\emph{1}}$\equiv E_2$. Finally,
system (\ref{compact-STI-pde}) is clearly monotone cooperative \cite{Smith1995}. Consequently,
the whole assumptions in \emph{(\textbf{L'})} are verified and the
following result is derived from Theorem \ref{Theorem-Fang-Zhao}, page \pageref{Theorem-Fang-Zhao}:

\begin{theorem}\label{Theorem-SIT-PDE-bistable-TW} Assuming that assumptions (\textbf{L'}) holds for
system (\ref{compact-STI-pde}) with $E_0=$\textbf{\emph{0}},
$\alpha\equiv E_1$, and \textbf{\emph{1}}$\equiv E_2$, then
system (\ref{compact-STI-pde}) admits a monotone wavefront $(U,c)$
with $U(-\infty)=$\textbf{0} and $U(+\infty)=E_2.$
\end{theorem}
\noindent Theorem \ref{Theorem-SIT-PDE-bistable-TW} holds true for system (\ref{SIT-pde}).

In the numerical simulation section, we will also discuss the case
where we take into account diffusion of sterile male mosquitoes.
That is, instead of model (\ref{SIT-pde}), page \pageref{SIT-pde},
we will consider system (\ref{SIT-pde-diffusion-MT}) with constant and non-constant continuous releases, that is
\begin{equation}\label{SIT-pde-diffusion-MT}
 \left\{%
\begin{array}{lcl}
 \displaystyle\frac{\partial A}{\partial t} &=& \phi F-(\gamma+\mu_{A,1}+\mu_{A,2}A)A, \quad (t,x)\in\mathbb{R}_+\times\mathbb{R}\\
    \displaystyle\frac{\partial M}{\partial t} &=& d_M\displaystyle\frac{\partial^2 M}{\partial x^2}+ (1-r)\gamma A-\mu_MM,\\
    \displaystyle\frac{\partial F}{\partial t} &=& d_F\displaystyle\frac{\partial^2 F}{\partial x^2}+\displaystyle\frac{ M}{M+M_T}r\gamma A-\mu_FF,\\
    \displaystyle\frac{\partial M_T}{\partial t} &=& d_T\displaystyle\frac{\partial^2 M_T}{\partial x^2}+ \Lambda(x,t)-\mu_TM_T,\\
\end{array}
\right.
\end{equation}

%and
%
%\begin{equation}\label{SIT-pde-diffusion-MT-Pulse}
% \left\{%
%\begin{array}{l}
% \displaystyle\frac{\partial A}{\partial t} = \phi F-(\gamma+\mu_{A,1}+\mu_{A,2}A)A, \quad (t,x)\in \cup_{n\in \mathbb{N}}(n\tau,(n+1)\tau)\times\mathbb{R}\\
%    \displaystyle\frac{\partial M}{\partial t} = d_M\displaystyle\frac{\partial^2 M}{\partial x^2}+ (1-r)\gamma A-\mu_MM,\\
%    \displaystyle\frac{\partial F}{\partial t} = d_F\displaystyle\frac{\partial^2 F}{\partial x^2}+\displaystyle\frac{ M}{M+M_T}r\gamma A-\mu_FF,\\
%    \displaystyle\frac{\partial M_T}{\partial t} = d_T\displaystyle\frac{\partial^2 M_T}{\partial x^2}-\mu_TM_T,\\
%    M_T(n\tau^+,x) = M_T(n\tau,x)+\tau\Lambda, \quad x\in\mathbb{R},\quad n=1,2,...
%\end{array}
%\right.
%\end{equation}

 where $\Lambda(x,t)$ is the number of sterile males released per
unit of time, $1/\mu_T$ is the average lifespan of sterile males. %, $\tau$ (in unit of time) is the pulse release period and $$\zeta(t^+)=\lim\limits_{\theta\rightarrow0^+}\zeta(t+\theta).$$
System (\ref{SIT-pde-diffusion-MT}) is considered with nonnegative
initial data. %However, we will assume in the sequel that for all
\section{Numerical simulations}\label{Numerical-simulations}
In this section we present some numerical simulations of system (\ref{compact-STI-pde})-(\ref{H-0}) to illustrate our analytical findings. Since we consider a one-dimensional model, the full discretization is simply obtained using a second-order finite difference method for the space discretization, and a first-order non-standard finite difference method for the temporal discretization, with the time-step following a CFL-condition, to preserve the positivity of the solution \cite{Anguelov2012,Dufourd2013}. 

%The scheme performs well and fast.
 
Following \cite{Anguelov2019}, we consider parameter values for \textit{Aedes albopictus}, summarized in Table \ref{Table-valeurs-parametre}, page \pageref{Table-valeurs-parametre}.

\begin{table}[H]
  \centering
\begin{tabular}{|c|c|c|c|c|c|c|c|c|c|c|}
  \hline
  % after \\: \hline or \cline{col1-col2} \cline{col3-col4} ...
  Symbol & $\phi$ &  $\mu_{A,1}$ & $\mu_{A,2}$ & $r$ & $\gamma$ & $\mu_F$ & $\mu_M$ & $\mu_T$ & $d_F$ & $d_M$\\
  \hline
  Value & 10 & 0.05 & 2$\times10^{-4}$ & 0.49 & 0.08 & 0.1 & 0.14 & 0.14 & 0.1 & 0.05\\
  \hline
\end{tabular}
 \caption{Entomological parameter values \cite{Anguelov2019,Dufourd2013} for \textit{Aedes albopictus}}\label{Table-valeurs-parametre}
\end{table}

%\begin{table}[H]
%  \centering
%\begin{tabular}{|c|c|c|c|c|c|c|c|c|c|c|}
%  \hline
%  % after \\: \hline or \cline{col1-col2} \cline{col3-col4} ...
%  Symbol & $\phi$ &  $\mu_{A,1}$ & $\mu_{A,2}$ & $r$ & $\gamma$ & $\mu_F$ & $\mu_M$ & $\mu_T$ & $d_F$ & $d_M$\\
%  \hline
%  Value & 6 & 0.0243 & 2$\times10^{-4}$ & 0.485 & 1/24.6 & 1/75.1 & 1/86.4 & 1/86.4 & 0.05 or 0.1 & 0.05\\
%  \hline
%\end{tabular}
% \caption{Entomological parameter values \cite{Aronna2020,Ekesi2006} for \textit{Bactrocera dorsalis}}
% \label{Table-valeurs-parametre-Bd}
%\end{table}
However, whatever the biological example, the next simulations are mainly for discussions and illustrations, even if we try to highlight some results for potential application in the field.
In addition, using the numerical schemes, we will go further by extending to a spatial domain, a corridor for instance, the small-massive releases strategy developed and studied in \cite{Anguelov2019}.

For sake of clarity, in this section when we speak about constant
release $M_T$ reader should understand that the effective amount of
the sterile insect released is $\Lambda=M_T\times\mu_T$.
Here, thanks to the parameters values given in Table \ref{Table-valeurs-parametre}, we have $\mathcal{R}\approx 30>1$. Based on \eqref{MT1}, we can estimate $M_{T_{1}} \approx 3745$.

First, we provide some simulations without SIT control. Then, we present some simulations with SIT control, exploring different strategies that can be used to eliminate, slow down, block, or even reverse a pest/vector invasion.

\subsection{Without SIT control - $M_T=0$}
Here, we assume that there is no SIT control, i.e.
$M_T=0$. In Fig. \ref{mini_speed_gamma_dF}(a), page \pageref{mini_speed_gamma_dF}, we show the variations of the minimal wave speed $\overline{c}$, estimated using (\ref{c-bar}), page \pageref{c-bar}, according to the maturation rate from larvae to adult,
$\gamma$, and the female mosquitoes diffusion rate, $d_F$. According to the parameter values given in Table \ref{Table-valeurs-parametre}, page \pageref{Table-valeurs-parametre}, we consider $\gamma>\dfrac{\mu_{A,1}\mu_F}{r\phi-\mu_F}$, such that we always have $\mathcal{R}>1$. No surprise in this figure: the larger $\gamma$, the larger the velocity; this comes from the fact that larvae emerge faster as adults such that the TW velocity speed-up: this case occurs when the environmental conditions are optimal for the larvae development; in contrary, when the temperature is low, the maturation rate slow down and thus $\gamma$

\begin{figure}[h!]
    \centering
\includegraphics[width=0.49 \linewidth]{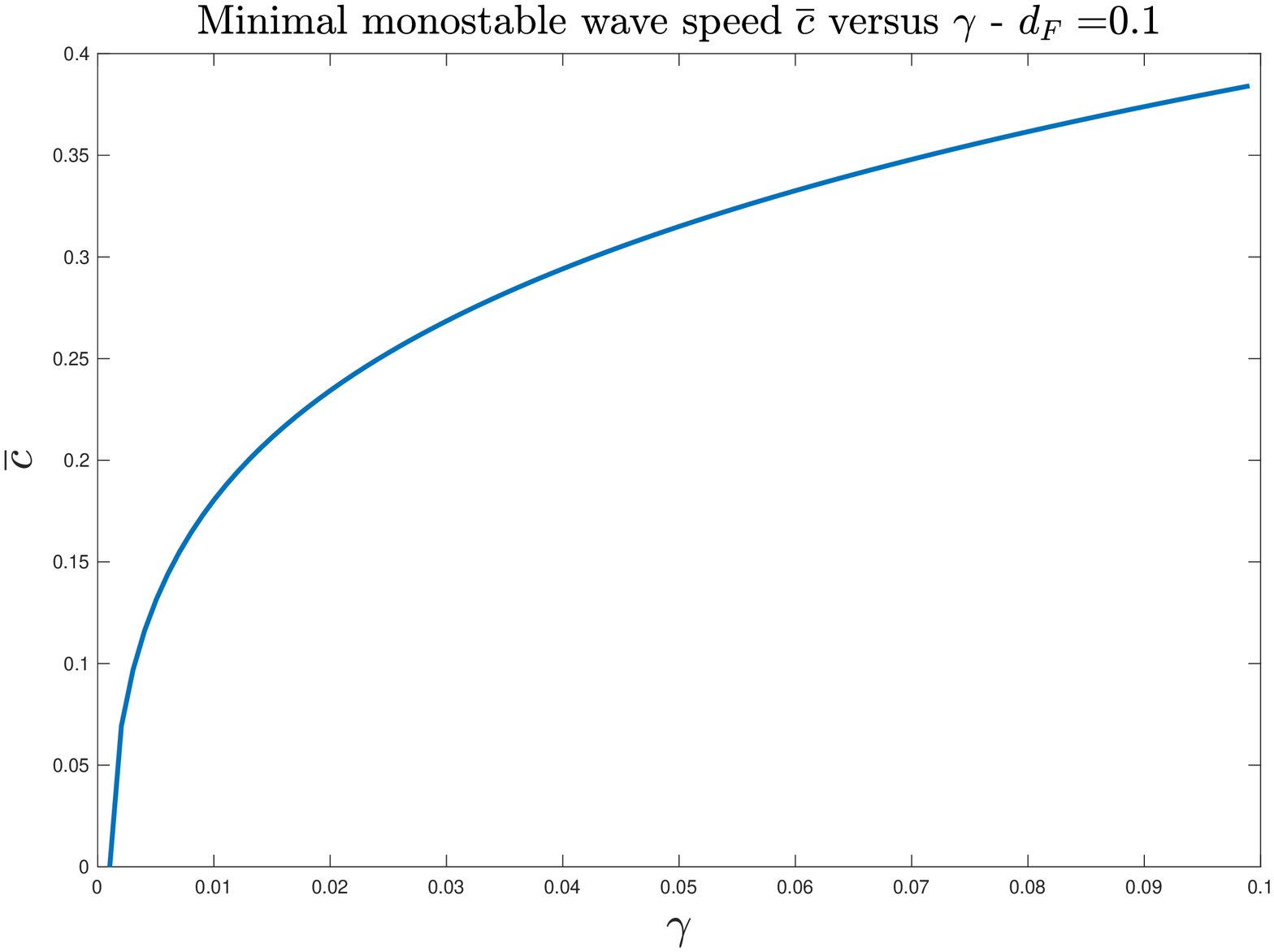}
\includegraphics[width=0.49 \linewidth]{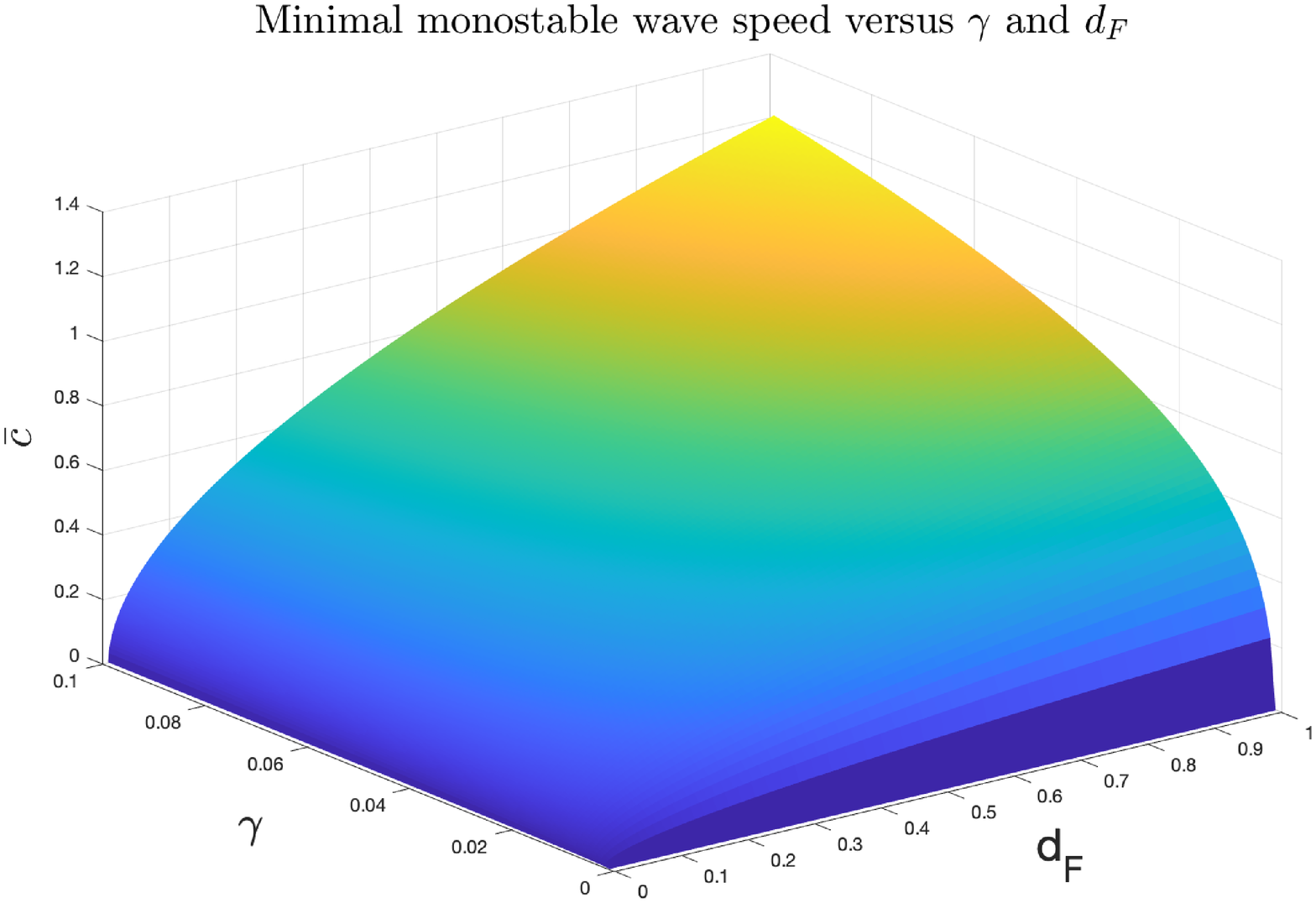}
    \caption{Variations of the minimal wave speed $\overline{c}$ of the monostable traveling wave solution of
    system (\ref{compact-STI-pde})-(\ref{H-0}) connecting the elimination equilibrium \textbf{0} and the positive
    equilibrium $E^*$ when no SIT occurs, i.e. $M_T=0$. (a): the evolution of $\overline{c}$ versus $\gamma$ for a given value of $d_F$, namely $d_F=0.1$; (b): the evolution of $\overline{c}$ versus $\gamma$ and $d_F$.} 
    \label{mini_speed_gamma_dF}
\end{figure}

In Fig. \ref{tw_mono}, page \pageref{tw_mono}, we represent the invasive monostable wave
solution of system (\ref{compact-STI-pde})-(\ref{H-0}) when there is
no SIT control. Only the immature stage (Fig. \ref{tw_mono}-(a))
and fertilized and eggs-laying female (Fig. \ref{tw_mono}-(b)) mosquito components are
displayed. The wave connects the unstable elimination equilibrium
$\textbf{0}$ and
the stable wild equilibrium, ${\bf{E}}^*\approx(18950,5412,7429)'$. Starting
with a local distribution of wild mosquitoes, transient states first take place (e.g. at times 25, 50, 75 and 100 days in figure
\ref{tw_mono}). After these transient states, the invasive
monostable wave occurs (e.g. at times 125, 150, 175 and 200 days in Fig.
\ref{tw_mono}, page \pageref{tw_mono}). Thanks to the previous estimate, the speed of the monostable traveling wave is $\overline{c}\approx 0.362$ km/day. In the long term dynamic, we have a complete invasion of the spatial domain (e.g. between times 375 and 400 days in figure \ref{tw_mono}).

\begin{figure}[h!]
    %\centering
    \hspace{-0.5cm}
    \subfloat[][Snapshots of immature stage dynamics]{  \includegraphics[scale=0.57]{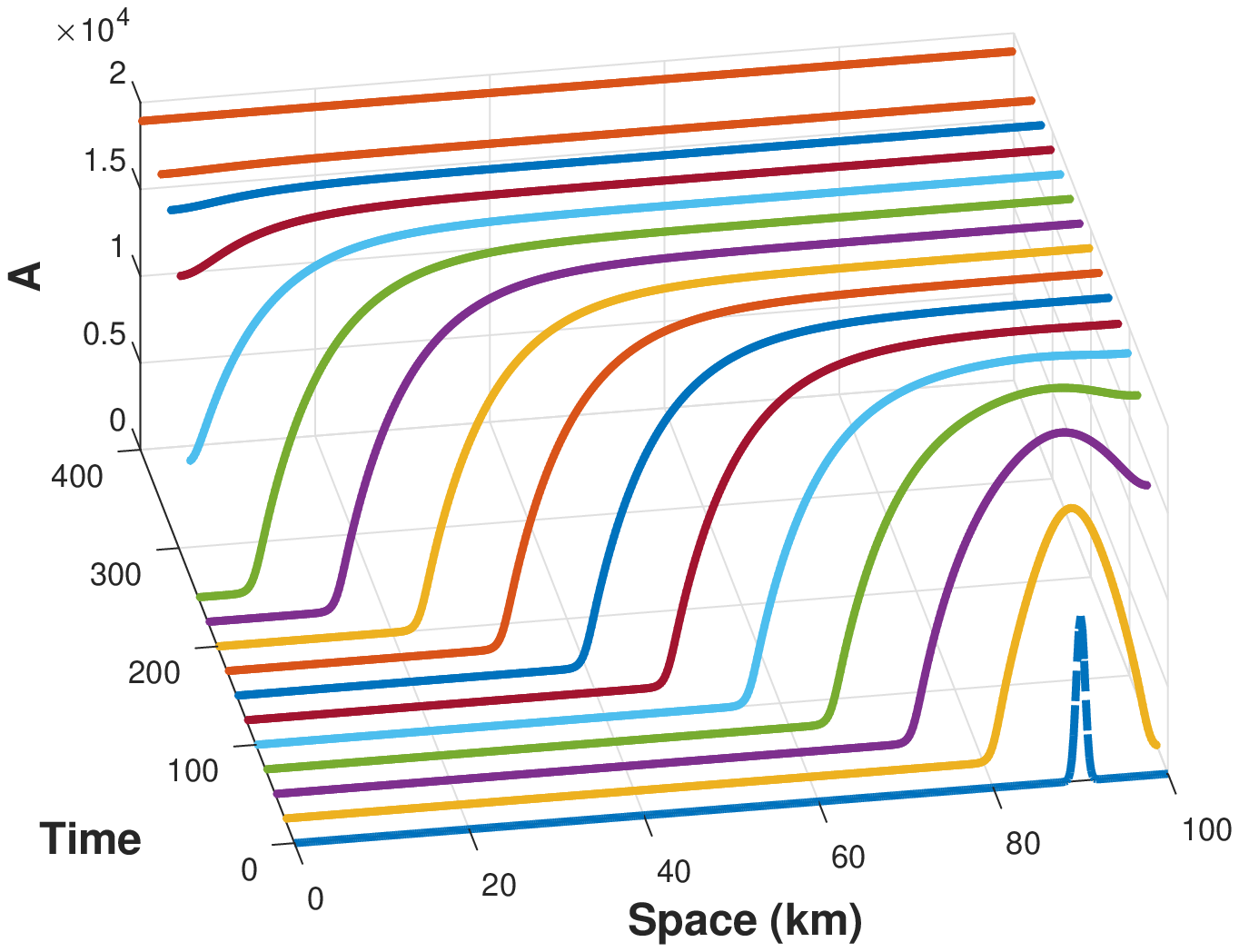}}
  %  \vspace{0.2cm}
    \subfloat[][Snapshots of fertilized female dynamics]{  \includegraphics[scale=0.57]{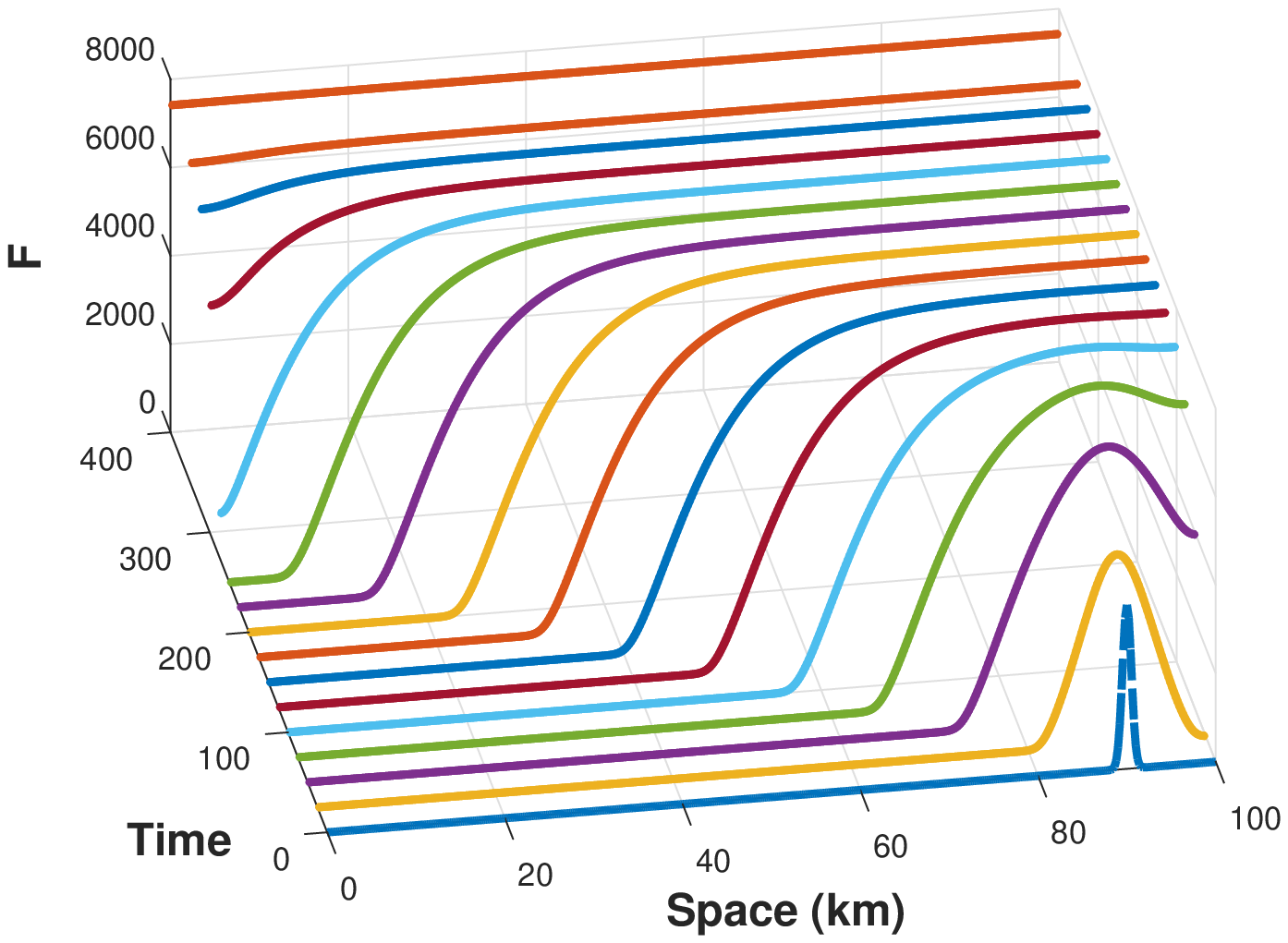}}
    \caption{Invasive monostable wave solution when there is no SIT release. The wave connects the unstable elimination equilibrium, $\textbf{0}$, to the stable wild equilibrium, ${\bf {E}}^*$. %The basic offspring number $R=30.15$. $\gamma=0.08$, $d_F=0.1$ and $d_M=0.05$. The rest of the parameters are given in Table\ref{Table-valeurs-parametre}.
    }
    \label{tw_mono}
\end{figure}

In the next section, we assume that SIT releases are considered. We
further assume two initial configurations: a partial invasion of the
spatial domain and a full invasion of the spatial domain.

\subsection{With SIT control - $M_T>0$}

%\subsubsection{Invasive wave and elimination wave}
%From now, we assume that, without the SIT control, an  invasive monostable traveling exists. This monostable wave involves the stable wild equilibrium $E^*$ and the unstable elimination equilibrium $\textbf{0}$ (see Fig. \ref{tw_mono}, page \pageref{tw_mono}). In the sequel, 
We now consider the monostable wave solution (e.g. $t=170$
days in Fig. \ref{tw_mono}, page \pageref{tw_mono}) as the initial setting. Therefore,
there exists a vector/pest-free subdomain and a vector/pest-persistent
subdomain.  We introduce sterile males, $0<M_T<M_{T_1}$, such that $\bf{0}$ becomes LAS, introducing a strong Allee effect.
We know that Allee effects can slow or even reverse traveling wave solutions: this is exactly what we observe in the following simulations.

In Figs. \ref{tw_serie_1}(a-b-c), page \pageref{tw_serie_1}, we observe that the introduction of sterile males, i.e. $M_T=1000, 2000, 3000$ respectively, slows down the traveling wave speed and thus the invasion: see also Fig. \ref{fig4}, page \pageref{fig4}, where numerical estimates of the traveling wave speed are provided. This shows that even if the SIT-threshold is not reached but the amount of sterile males to be released is sufficient, it can help to delay a pest/vector invasion.

However, there exists a critical value $M_T^c$, close to $M_T$, such that the traveling wave stops or reverses leading to elimination in the (very) long term (see Fig. \ref{tw_serie_1}(d) and Fig. \ref{fig4}). However, since $M_{T}^c$ is close to $M_{T_1}$, it is more interesting to release above $M_{T_1}$ since we know that $\bf{0}$ is GAS, such that the system will reach elimination more or less quickly depending on whether $M_T$ is larger or much larger than $M_{T_1}$, like $M_T=k\times M_{T_{1}}$ with $k>1$. 

To avoid the permanent use of massive releases, we can also use the massive and small releases strategy developed in \cite{Anguelov2019}. Indeed, when $M_T>0$, the SIT problem has three equilibria, $\bf{0}$, $\bf{E}_1$, and $\bf{E}_2$, where $\bf{E}_1$ is unstable while $\bf{0}$ and $\bf{E}_2$ are LAS, and such that $[\bf{0},\bf{E}_1)$ belongs to the basin of attraction of $\bf{0}$ and $\bf{E}_1$ is defined for a given value of $M_T$, say $100$ \cite{Anguelov2019}. 
Thus, the massive-small strategy consists of releasing first sterile males massively, i.e. $M_T=k\times M_{T_1}$, with $k>1$, in order to reach the parallelepiped $[\bf{0},\bf{E}_1)$. Then, once $[\bf{0},\bf{E}_1)$ is reached, we can switch from massive releases to small releases, $M_T=100$, and use the AS of $\bf{0}$ within $[\bf{0},\bf{E}_1)$ to drive the wild population to elimination. This is feasible but it requires to release the sterile males homogeneously over the treated area.

%\tcb{We consider two reference times (e.g. $t_1=5000$ days and $t_2=6000$ days) and we numerically estimated, see also \cite{Banasiak2019}, the speed of the invasive wave (resp. the elimination wave) and found $c=0.0040$ km/day (resp. $c=0.0021$ km/day). 
%In addition, when varying the time-step used for the temporal discretization, the tipping point $M_T^c$ still exist. Recall that when $M_T>M_{T_1}$, $\textbf{0}$ is the unique homogeneous equilibrium of system (\ref{compact-STI-pde})-(\ref{H-0}) and it is globally asymptotically stable. Hence, having an ``elimination" traveling wave when $M_T$ is ``close" to $M_{T_1}$ is not unrealistic.}
%Therefore, in Fig. \ref{tw_serie_1}, page \pageref{tw_serie_1}, where we
%represented only immature stage and female mosquitoes, the SIT
%release done in the whole spatial domain is able to slow down and
%even reverse the orientation of the bistable wave. Moreover, the
%final qualitative behavior (i.e. elimination or invasion) takes
%place, slowly, thanks to the SIT release.

However, the previous strategies are unrealistic in the field: it is impossible to treat large areas using (permanent) massive releases. Also, some areas may be difficult to reach in order to release sterile males homogeneously. In general a barrier or a corridor strategy is recommended, but the difficulty is to define the width of this corridor and also the right strategy to avoid the risk of pest/mosquito emergence within the untreated area that has to be protected.

\begin{figure}[h!]
    \centering

    %\hspace{-1.5cm}
    \subfloat[][Bistable Invasion - $M_T=1000$]{\includegraphics[width=0.49 \linewidth]{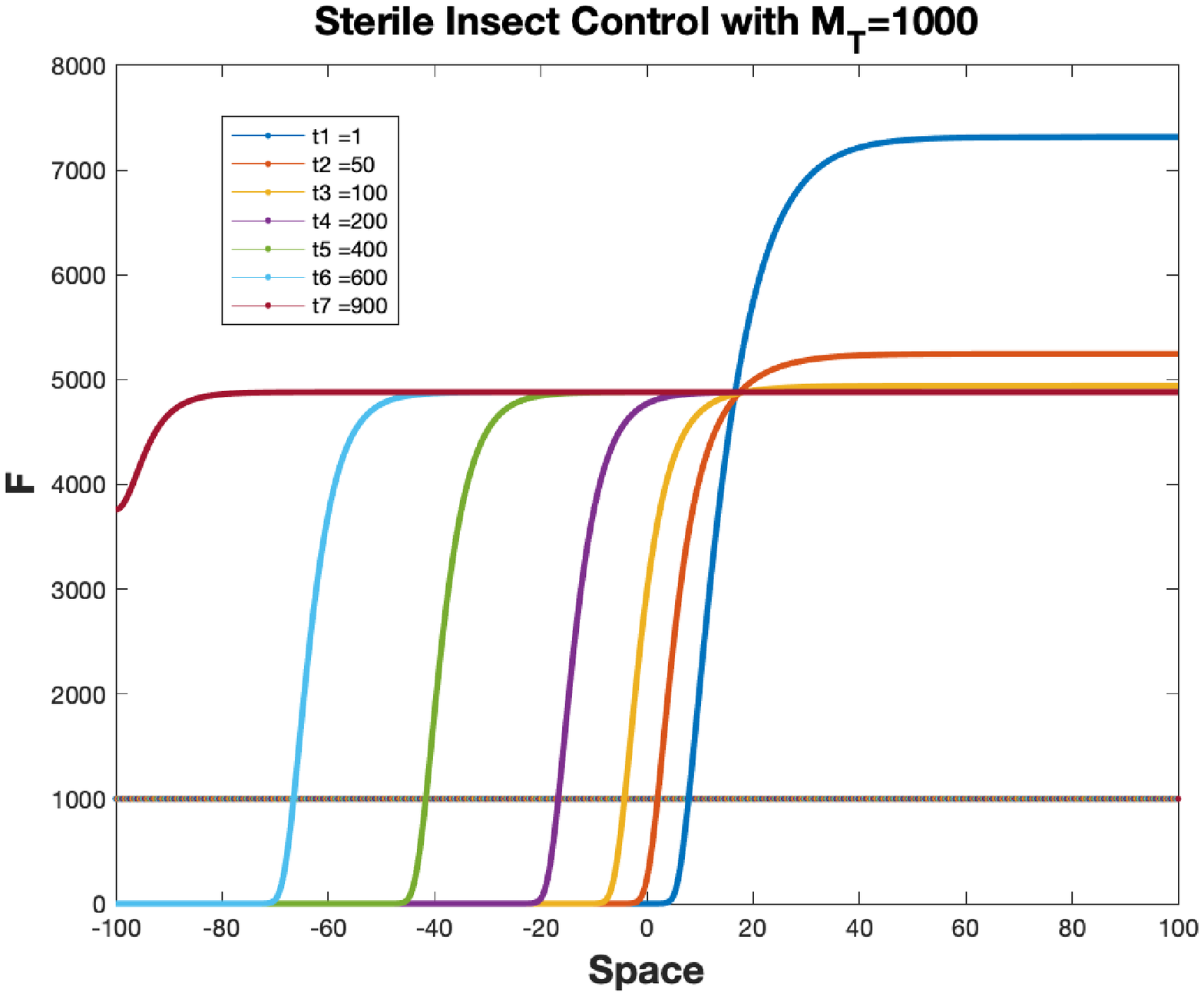}}
    \subfloat[][Bistable Invasion - $M_T=2000$]{\includegraphics[width=0.49 \linewidth]{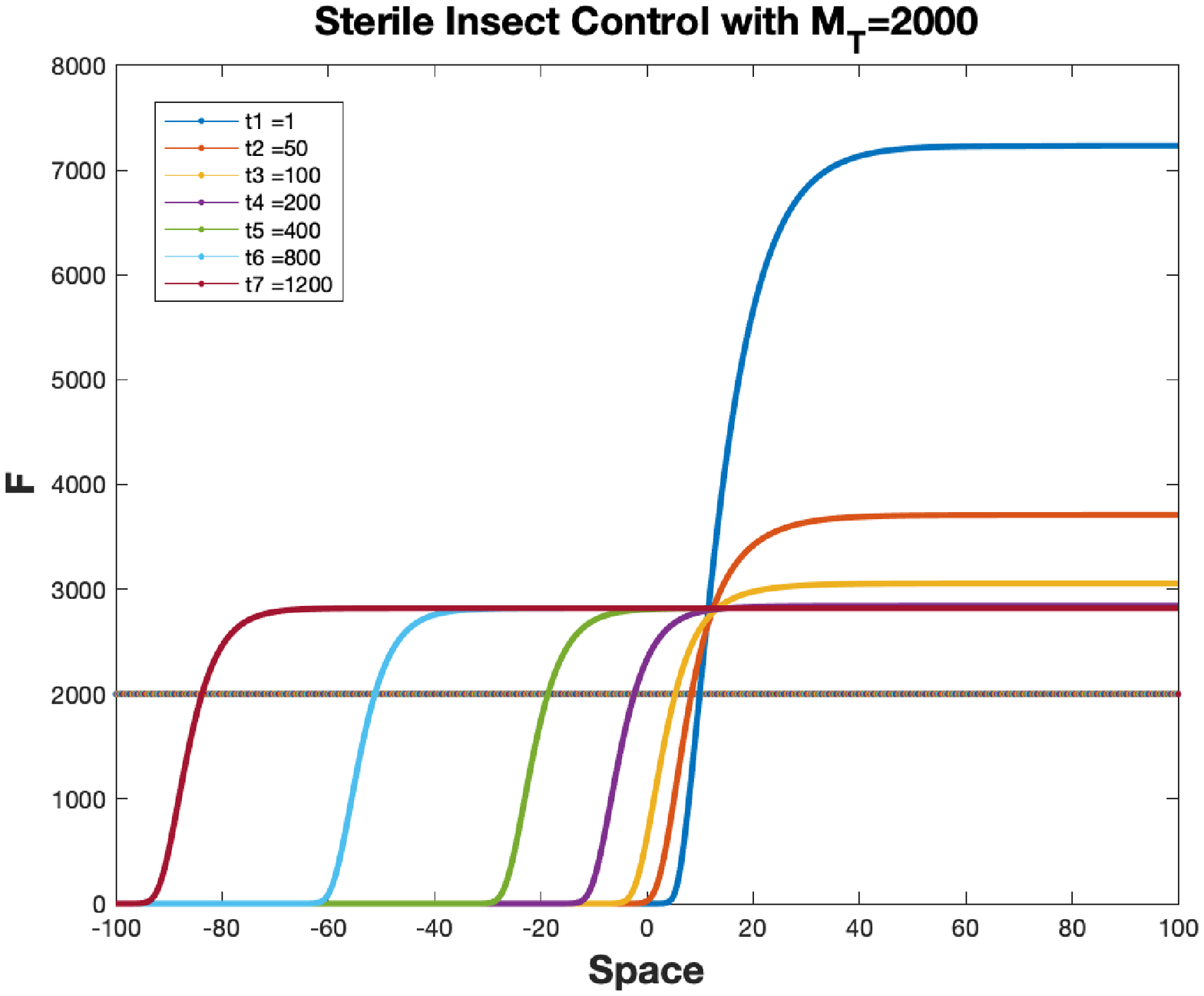}}
    
    \vspace{0.2cm}
    \subfloat[][Bistable Invasion - $M_T=3000$]{\includegraphics[width=0.49 \linewidth]{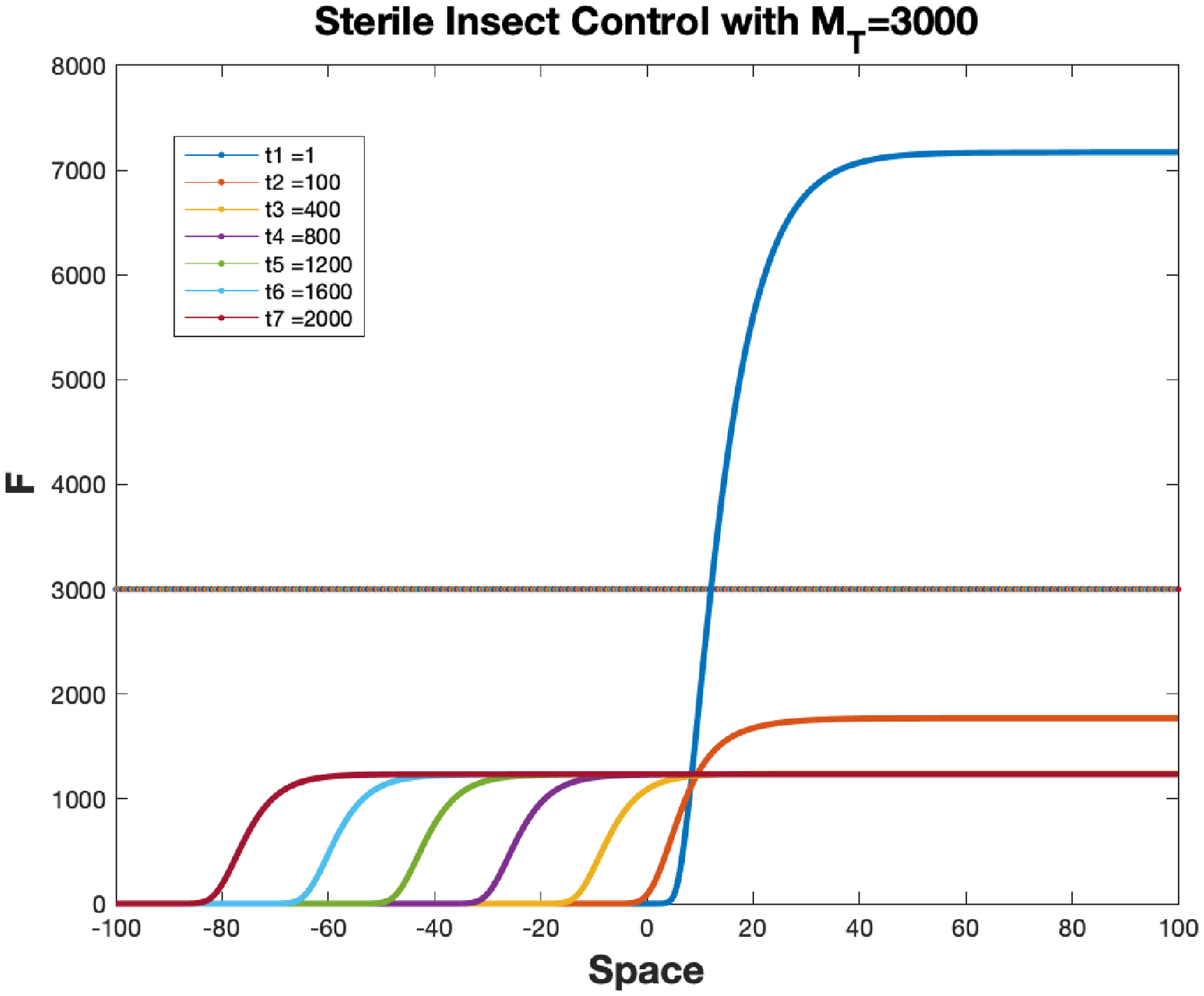}}
    \subfloat[][Bistable Elimination - $M_T=3720$]{\includegraphics[width=0.49 \linewidth]{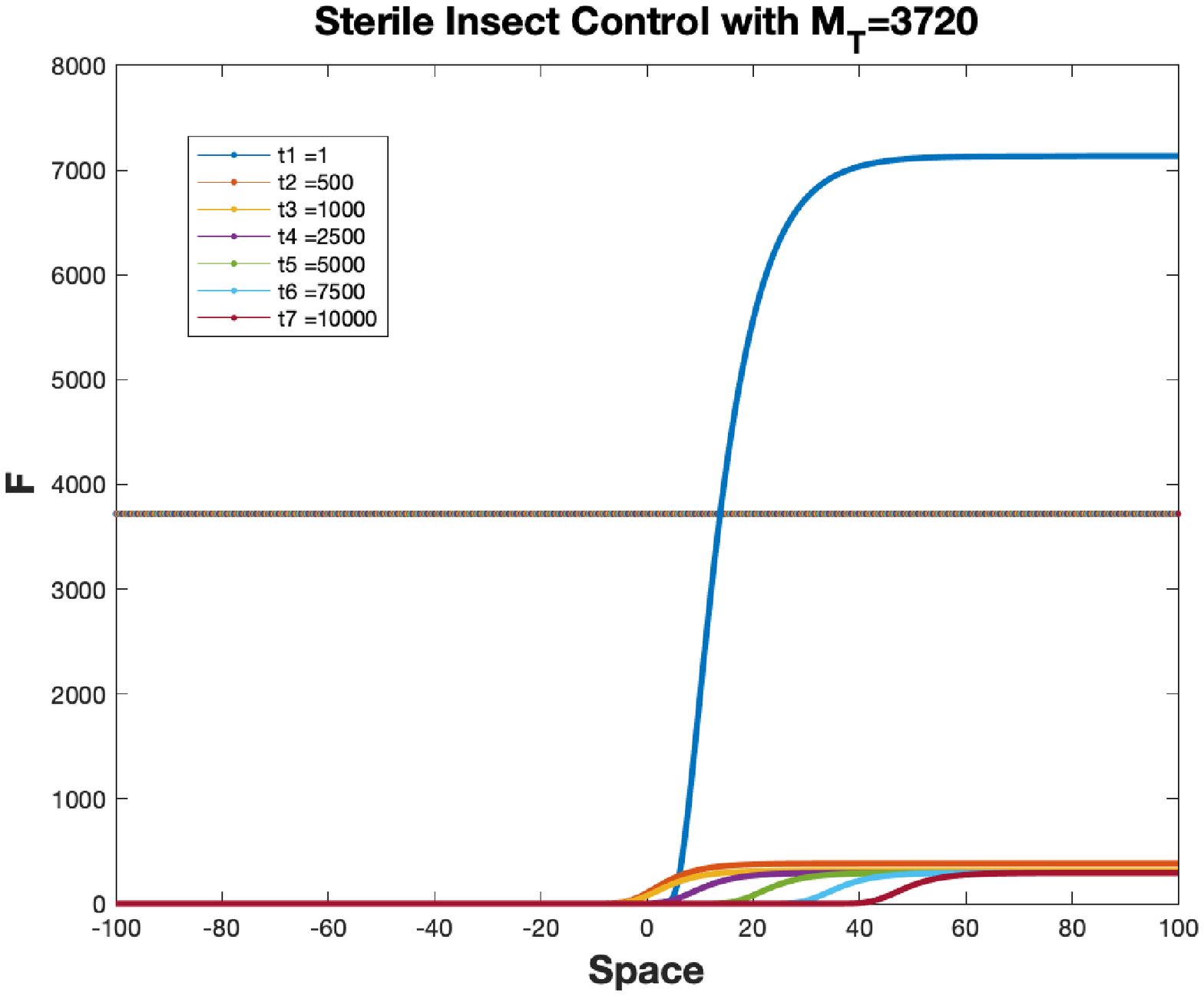}}

    \caption{Long term dynamics for several values of SIT releases when $0<M_T<M_{T_1}$}
    \label{tw_serie_1}
\end{figure}

\begin{figure}[H]
    \centering
\includegraphics[width=0.75 \linewidth]{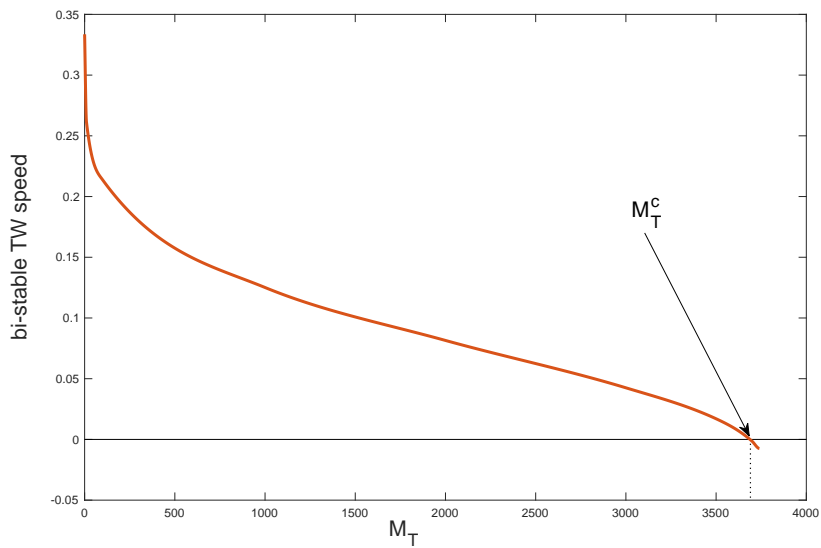}

    \caption{Numerical estimates of the bistable traveling wave speed when $0<M_T<M_{T_1}$}
    \label{fig4}
\end{figure}

\subsection{The ``corridor/barrier strategy" - blocking wave}
Following \cite{Lewis1993, Seirin2013b}, we want to apply the control strategy developed and studied in \cite{Anguelov2019}, combining massive releases (locally) and small releases to stop the invasion of mosquitoes or pest and eventually to push them back. In fact we use the location of the front linking $\bf{0}$ and $\bf{E}_2$, to define a sufficiently large corridor, where massive releases, $M_T>M_{T_1}$, will be used, while the free area will be treated with small releases only.

Here, consider that we want to protect an area delimited by $0$ and $x_{\min}$ from an invasion of pest/mosquitoes using SIT releases. To do that we release a massive number of sterile males within corridor  $[x_{\min},x_{\max}]$ in order to block the invasive front. 
Within $[x_{\min},x_{\max}]$, we release $1.1\times M_{T_1}$ sterile males, while in $[0,x_{\min})$, we release only $M_{T_1}/100$ sterile males, which leads to define the unstable equilibrium. When $d_F=0.05$ or $d_F=0.1$, we choose $x_{\min}$ and $x_{\max}$ such that $x_{\max}-x_{\min}=20$ km: see Figs. \ref{blocking1} and \ref{blocking2}, page \pageref{blocking1}. Thus clearly, for the same amount for massive releases, for different, but close, values of $d_F$, the starting dates of the corridor control is crucial: $300$ for $d_F=0.05$, and $200$ for $d_F=0.1$. As seen, the strategy we have developed in \cite{Anguelov2019} is still very suitable when considering the spatial component.

\begin{figure}[h!]
    \centering
 \includegraphics[scale=0.4]{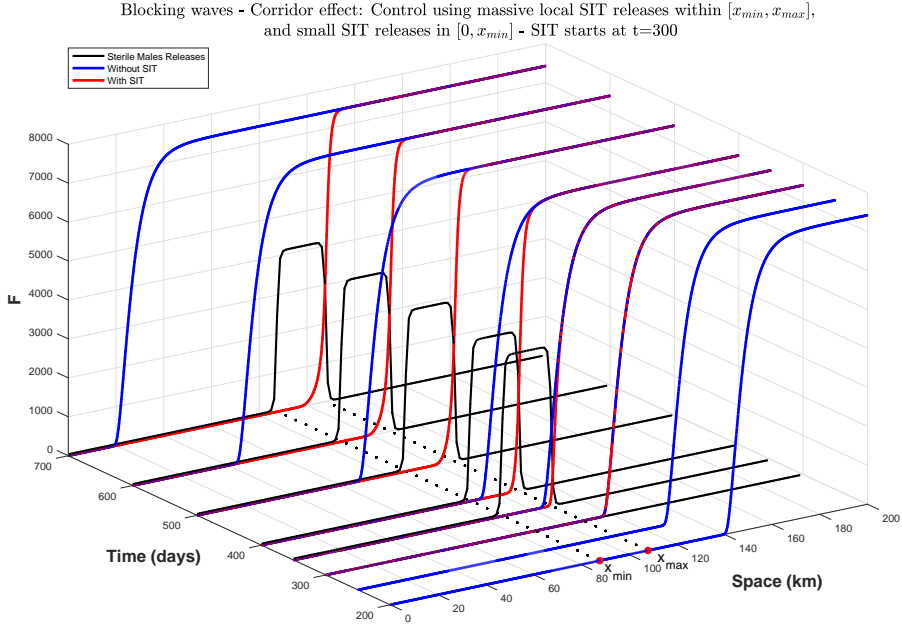}
    \caption{Control strategy using a 20 km corridor to block mosquito/pest invasion using massive releases within the corridor, and small releases on the left side of the corridor. SIT starting date: $t=300$. We consider 
     $\gamma=0.08$, $d_F=0.05$, $d_M=0.05$, $M_T=1.1\times M_{T_1}$. The rest of the parameters are given in Table
\ref{Table-valeurs-parametre}. }
    \label{blocking1}
\end{figure}

\begin{figure}[h!]
    \centering
 \includegraphics[scale=0.4]{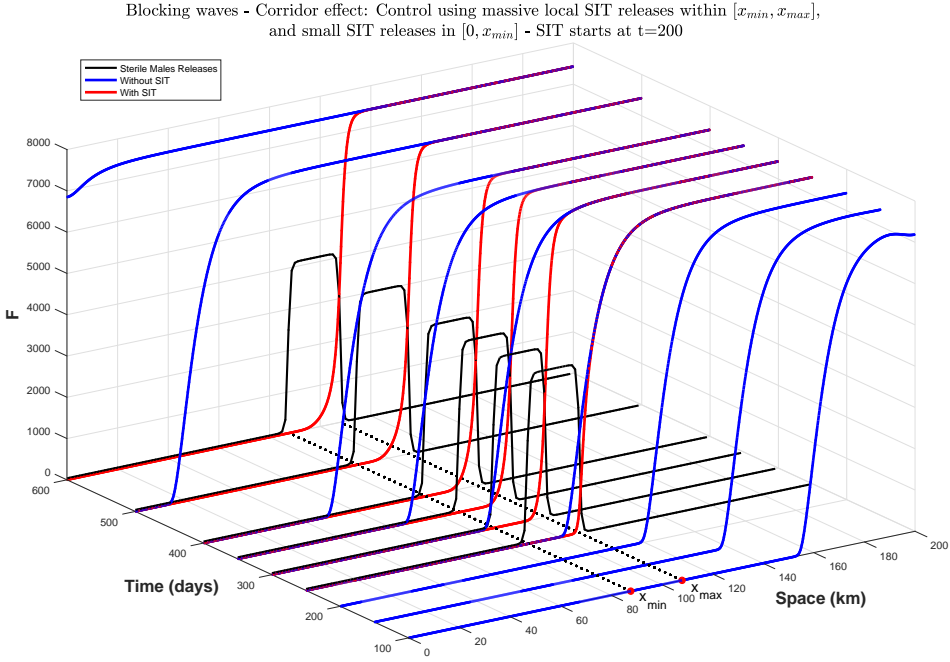}
    \caption{Control strategy using a 20 km corridor to block mosquito/pest invasion using massive releases within the corridor, and small releases on the left side of the corridor. SIT starting date: $t=200$. We consider 
     $\gamma=0.08$, $d_F=0.1$, $d_M=0.05$, $M_T=1.1\times M_{T_1}$. The rest of the parameters are given in Table
\ref{Table-valeurs-parametre}. }
    \label{blocking2}
\end{figure}

The same strategy can be used to stop and push back the wild pest/mosquito wave. Indeed, once the wave is blocked, it might be possible to move to the right the corridor, like a ``traveling carpet". Using the same parameters and also the same corridor width, but releasing $1.8\times M_{T_1}$ sterile males inside $[x_{\min}(t),x_{\max}(t)]$ we obtain the simulations in Fig. \ref{pushing}, page \pageref{pushing}. The dotted lines indicate the original location of the control to block the wave. Then, the corridor is moved once, at the left-end, $x_{\min}(t)$ of the corridor, we have $(E,M,F)'<(E_1,M_1,F_1)'$.

It is possible to push back faster the mosquito/pest traveling wave by expanding the corridor width from $20$ km to $40$ km (see Fig. \ref{pushing2} page \pageref{pushing2}), for instance. The counterpart is, of course, to release a larger number of sterile males within the corridor. This shows that a relationship might exists between the size of the massive releases, the corridor depth and the duration of the massive releases. These parameters can also be constrained by the sterile males production constraint.

Last but not least, the same strategy could be used to treat a whole domain still invaded by a pest or mosquito and to get ride of them after longtime treatment. Of course, in order to avoid a re-invasion, it will be necessary to continue the SIT treatment, with small releases, in order to maintain the wild population under a given threshold, $\bf{E}_1$, for instance.

\begin{figure}[h!]
    \centering
 \includegraphics[scale=0.4]{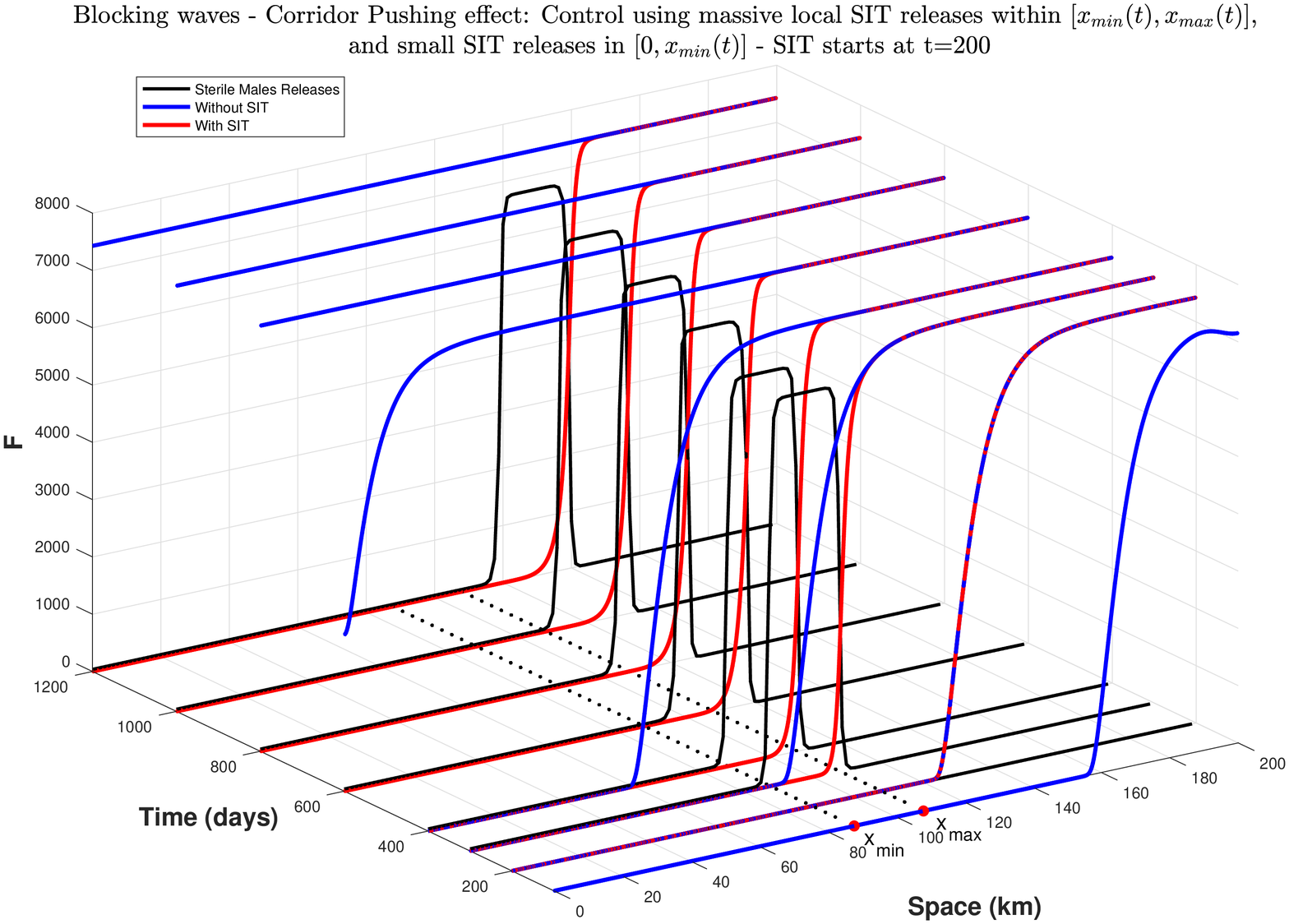}
    \caption{Dynamic Control strategy using a {\bf 27} km corridor to block and push back mosquito/pest invasion  using massive releases within the corridor, and small releases on the left side of the corridor. We consider 
     $\gamma=0.08$, $d_F=0.1$, $d_M=0.05$, and $M_T=1.8\times M_{T_1}$. The rest of the parameters are given in Table
\ref{Table-valeurs-parametre}. }
    \label{pushing}
\end{figure}

\begin{figure}[h!]
    \centering
 \includegraphics[scale=0.4]{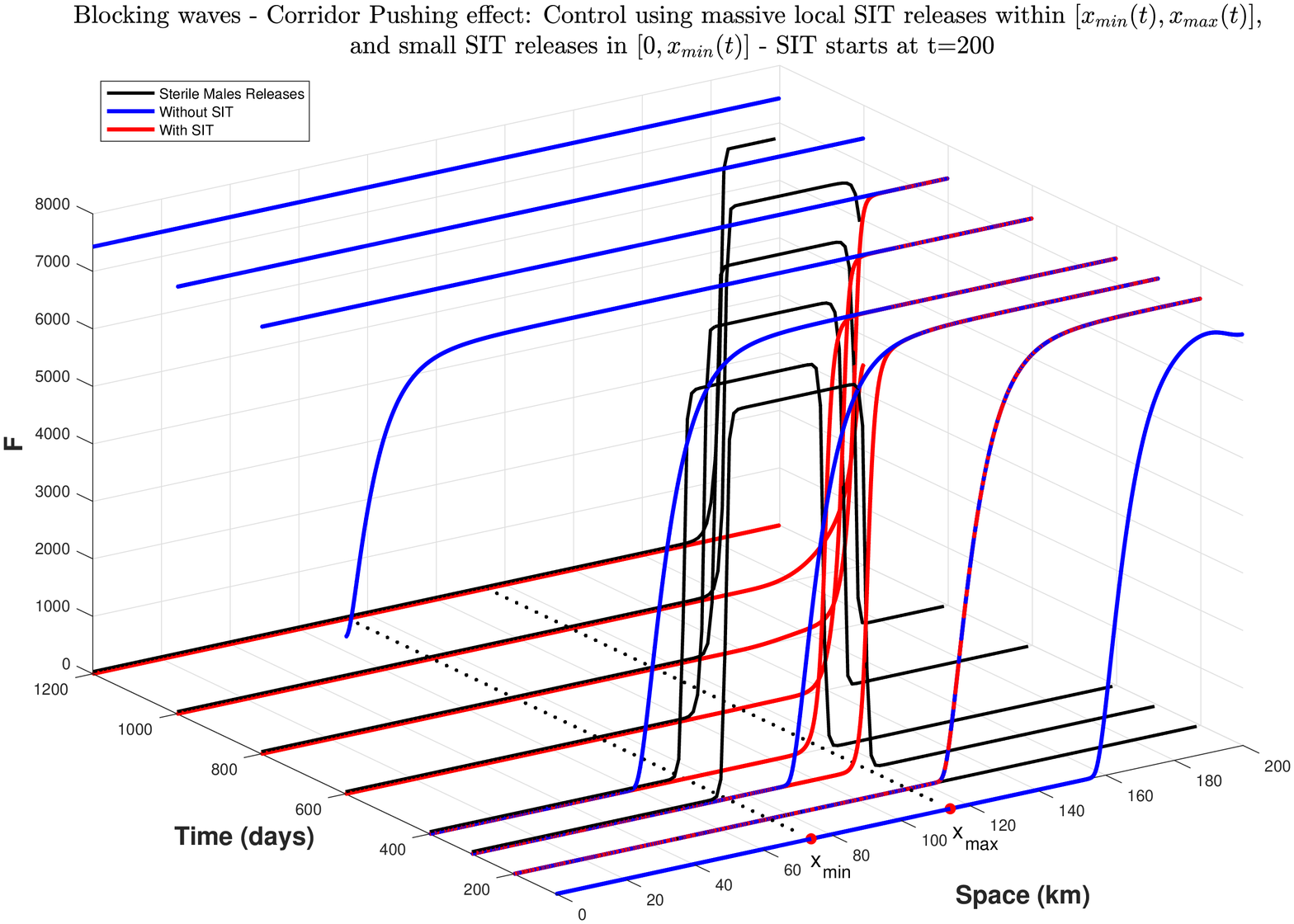}
    \caption{Dynamic Control strategy using a {\bf 40} km corridor to block and push back the invasion of mosquitoes using massive releases within the corridor, and small releases on the left side of the corridor. We consider 
     $\gamma=0.08$, $d_F=0.1$, $d_M=0.05$, $M_T=1.8\times M_{T_1}$. The rest of the parameters are given in Table
\ref{Table-valeurs-parametre}. }
    \label{pushing2}
\end{figure}

The previous ``corridor" and ``barrier" strategies could also be used against other pest, like \textit{Ceratitis capitata} (``medfly") or \textit{Bactrocera dorsalis} (``oriental fruit fly"). Indeed, \textit{C. capitata} was removed from invaded areas in southern Mexico, and for thirty years a sterile fly barrier across Guatemala has maintained Mexico and the USA medfly-free. The barrier strategy is still used successfully against screwworms, \textit{Cochliomyia hominivorax} (Coquerel), at the border of Panama and Colombia, protecting North-America from this damaging pest \cite{screwworm}. However, the objectives are now to use SIT more locally to protect or free (from pest) specific places instead of treating large area.

\section{Conclusion}
\label{Conclusion}
In this work, we have extend the temporal model studied in \cite{Anguelov2019} to a spatio-temporal one. In order to show the existence of traveling wave solutions we have also extended results from \cite{Fang2009}. This allows us to show that when $0\leq M_T<M_{T_1}$, on the whole area, a monotone wavefront may exist, connecting $\bf{0}$ with a positive equilibrium. When SIT control is such that $M_T^c<M_T<M_{T_1}$, on the whole area, then we showed, numerically, that the population can be driven to elimination. However, for realistic use, we have considered the strategy of massive-small releases studied in \cite{Anguelov2019}, using massive releases in a spatial corridor of sufficient depth to block the front of the traveling wave, and using small releases within the targeted area, i.e. the pest/vector free area. We show that this strategy can be successful to block the invasion and eventually can also be used, using a ``dynamic corridor", to push back the pest or the vector and thus free invaded area.

Theoretically, the corridor case is still an open problem. In particular, in future studies, it would be useful to derive a relationship between the width of the corridor, the size of the massive releases, the sterile insects manufacturing capacity, etc. Such a result would be useful in order to implement appropriate sterile male release strategies.

Last but not least, our work takes place within several SIT projects in la R\'eunion (a French tropical island, located in the Indian Ocean, east from Madagascar and southwest from Mauritius island) and Corsica (a French island in the Mediterranean Sea, southeast of France and west of Italy), for which field experiments, and in particular field releases, are expected. We hope that our results will help to improve the release strategies and thus the impact of SIT, in combination with other biocontrol strategies, using mechanical control, prophylaxis, etc.
\vspace{0.5cm}
\section*{Acknowledgments}

The authors are partially supported by the ``SIT feasibility project against \textit{Aedes albopictus} in Reunion Island", TIS 2B (2020-2021), jointly funded by the French Ministry of Health and the European Regional Development Fund (ERDF). The authors are (partially) supported by the DST/NRF SARChI Chair  in Mathematical Models and Methods in Biosciences and Bioengineering at the University of Pretoria (grant 82770). The authors acknowledge the support of the STIC AmSud project NEMBICA (2020–2021), n$^o$20-STIC-05. YD is also partially supported by the CeraTIS-Corse project, funded by the call Ecophyto 2019 (project n$^o$: 19.90.402.001), against \textit{Ceratitis capitata}. Part of this work is also done within the framework of the GEMDOTIS project (Ecophyto 2018 funding), that is ongoing in La R\'eunion. This work was also co-funded by the European Union Agricultural Fund for Rural Development (EAFRD), by the Conseil R\'egional de La R\'eunion, the Conseil D\'epartemental de La R\'eunion, and by the Centre de Coop\'eration internationale en Recherche Agronomique pour le D\'eveloppement (CIRAD). YD greatly acknowledges the Plant Protection Platform (3P, IBiSA)

%\section*{\textbf{References}}\addcontentsline{toc}{section}{References}
\bibliographystyle{plain}
\bibliography{bibliography}

\appendix

\section{Some results on reaction-diffusion (RD)
systems}\label{appendix-material}

Let us consider the abstract Cauchy problem: % abstract form, equations (\ref{RD})-(\ref{RD-CI}) read as
\begin{equation}\label{RD-2}
\left\{
 \begin{array}{l}
 \displaystyle\frac{d u}{d t} +A u= F(u),\\
 u(0)=u_0
 \end{array}
 \right.
\end{equation}
where for $B=C_{ub}(\mathbb{R})$,

\begin{equation}\label{RD-A}
\left\{
\begin{array}{l}
  D(A) = \{f\in B: f'' \quad
\mbox{exists},\quad f'\in B,\quad f''\in B\}, \\
  Au = -du''.
\end{array}
\right.
\end{equation}

\comment{
According to Theorem \ref{theoreme-McBride}, the operator $A$
defined in (\ref{RD-A}) generates on $B=C_{ub}(\mathbb{R})$ the
$C_0-$semigroup of Gauss-Weierstrass $\{S(t)\}_{t\geq0}$ defined by
(\ref{Gauss-semigroup}) with $S(0)=I$.

\begin{definition}\label{local-lipschitz}
Suppose $F$ is a nonlinear operator from a Banach space $B$ into
$B$. $F$ is said to satisfy the local Lipschitz condition if for any
positive constant $M > 0$, there is a positive constant $L_M$
depending on $M$ such that when $u, v\in B$, $\|u\|\leq M$ and
$\|v\|\leq M$, $$\|F(u)-F(v)\|_B\leq L_M\|u-v\|_B.$$
\end{definition}

For
}
Systems of reaction-diffusion equations assume also the form (\ref{RD-2})
where $u=(u_1, u_2,...,u_n)$, $F=(F_1, F_2,...,F_n)$, $d=diag(d_1,
d_2,...,d_n)$ with $d_i>0$, for $i=1,2,...,n$ and\\
$\displaystyle\frac{\partial^2 u}{\partial
x^2}:=\left(\displaystyle\frac{\partial^2 u_1}{\partial
x^2},\displaystyle\frac{\partial^2 u_2}{\partial
x^2},...,\displaystyle\frac{\partial^2 u_n}{\partial x^2}\right)$. The corresponding Gauss-Weierstrass $C_0-$semigroup\\ 
$S(t)=(S_1(t),S_2(t),...,S_n(t))$ is defined on a Banach space
$B=B_1\times ...\times B_n$ by $S_i(0)=I$ and for $t>0$
\begin{equation}\label{Gauss-semigroup-n}
\begin{array}{ccl}
  (S_i(t)f)(x) &=&\displaystyle\frac{1}{\sqrt{4\pi d_i
t}}\displaystyle\int_{\R}\exp\left(-\displaystyle\frac{(x-y)^2}{4d_i
t}\right)f(y)dy   \qquad (x\in \mathbb{R}, f\in B_i) \\
   & = & (G_{t,i}\star f)(x)
\end{array}
\end{equation}
where $$G_{t,i}(x)=\displaystyle\frac{1}{\sqrt{4\pi d_i
t}}\exp\left(-\displaystyle\frac{x^2}{4d_i t}\right), \quad
i=1,2,...,n.$$ %Here, for $i=1,2,...,n$, $B_i= C_{ub}(\mathbb{R})$
%and for $g=(g_1,g_2,...,g_n)\in B$,
%$$\|g\|_B=\sum\limits_{1\leq i\leq n}\|g_i\|_{B_i}.$$

%\begin{remark}\label{Rauch-and-Smoller}
%Existence result is also given by Rauch and Smoller (1978)
%\cite{Rauch1978}, Theorem 2.1, including
In the case where some (but not all) diffusion coefficient
$d_{i_0}=0$, the corresponding operator in the Gauss-Weierstrass
$C_0-$semigroup is taken as (Rauch and Smoller (1978)
\cite{Rauch1978}) $S_{i_0}(0)=I$ and for $t>0$

\begin{equation}\label{Gauss-Rauch-Smoller}
\begin{array}{ccl}
  (S_{i_0}(t)f)(x) & = & (G_{t,i_0}\star f)(x)=f(x)
\end{array}
\end{equation}

with $G_{t,i_0}(x)=\delta(x)$. $\delta$ is the Dirac distribution
and is defined by

$$
  \delta(x)= \left\{\begin{array}{cl}
                      0, & x\neq0 \\
                      +\infty, & x=0
                    \end{array}
  \right.
$$
with $$\displaystyle\int_\mathbb{R}\delta(y)dy=1.$$
%\end{remark}

\section{Proof of Lemma \ref{boundeness-lemma}}\label{AppendixA0}
Recall that $w=(A,M,F)'$ and $$H(w)=\left(\phi F-(\gamma+\mu_{A,1}+\mu_{A,2}A)A,(1-r)\gamma A-\mu_MM, \displaystyle\frac{ M}{M+M_T}r\gamma A-\mu_FF\right)'.$$ Following \cite[pages 210-211]{Smoller1983}, one has
$$
\begin{array}{l}
G=-A, \quad \left.\nabla G\cdot H\right|_{A=0}=-\phi F\leq0 \quad\mbox{in}\quad \Gamma_{\rr\leq1}\cup\Gamma_{\rr>1}, \quad \mbox{so}\quad A\geq0.\\
G=-M, \quad \left.\nabla G\cdot H\right|_{M=0}=-(1-r)\gamma A\leq0 \quad\mbox{in}\quad \Gamma_{\rr\leq1}\cup\Gamma_{\rr>1}, \quad \mbox{so}\quad M\geq0.\\
G=-F, \quad \left.\nabla G\cdot H\right|_{F=0}=-\dfrac{M}{M+M_T}r\gamma A\leq0 \quad\mbox{in}\quad \Gamma_{\rr\leq1}\cup\Gamma_{\rr>1}, \quad \mbox{so}\quad M\geq0.\\
\end{array}
$$
\begin{enumerate}
	\item Assume that $\rr\leq1$ or equivalently $\dfrac{r\gamma}{\mu_F}\leq \dfrac{\gamma+\mu_{{A,1}}}{\phi}$. 
	
	$$
	\begin{array}{l}
	G=M-\dfrac{(1-r)\gamma}{\mu_M}k_1, \quad \left.\nabla G\cdot H\right|_{M=\dfrac{(1-r)\gamma}{\mu_M}k_1}=(1-r)\gamma(A-k_1)\leq0 \quad\mbox{in}\quad \Gamma_{\rr\leq1},\\
	 \quad \mbox{so}\quad M\leq \dfrac{(1-r)\gamma}{\mu_M}k_1.\\
	G=F-\dfrac{r\gamma}{\mu_F}k_1, \quad \left.\nabla G\cdot H\right|_{F=\dfrac{r\gamma}{\mu_F}k_1}\leq r\gamma (A-k_1)\leq0 \quad\mbox{in}\quad \Gamma_{\rr\leq1}, \quad \mbox{so}\quad F\leq \dfrac{r\gamma}{\mu_F}k_1.\\
	G=A-k_1, \quad \left.\nabla G\cdot H\right|_{A=k_1}\leq\phi\left(F-\dfrac{\gamma+\mu_{{A,1}}}{\phi}k_1\right)\leq\phi\left(F-\dfrac{r\gamma}{\mu_F}k_1\right)\leq0 \quad\mbox{in}\quad \Gamma_{\rr\leq1},\\ \quad \mbox{so}\quad A\leq k_1.\\
	\end{array}
	$$
	\item Assume that $\rr>1$. Hence the positive wild equilibrium $E^*=(A^*,M^*,F^*)'$ is defined. Note that 
	$M^*=\dfrac{(1-r)\gamma}{\mu_M}A^*$ and $F^*=\dfrac{r\gamma}{\mu_F}A^*$.	
	One has
	$$
	\begin{array}{l}
	G=M-k_2M^*, \quad \left.\nabla G\cdot H\right|_{M=k_2M^*}=(1-r)\gamma(A-k_2A^*)\leq0 ~ \mbox{in} ~ \Gamma_{\rr>1},
	 ~ \mbox{so} ~ M\leq k_2M^*.\\
	G=F-k_2F^*, \quad \left.\nabla G\cdot H\right|_{F=k_2F^*}\leq r\gamma (A-k_2A^*)\leq0 \quad\mbox{in}\quad \Gamma_{\rr>1}, \quad \mbox{so}\quad F\leq k_2F^*.\\
	G=A-k_2A^*, \quad \left.\nabla G\cdot H\right|_{A=k_2A^*}=\phi\left(F-(\gamma+\mu_{{A,1}}+\mu_{A,2}A^*)A^*\right)\\
	\hspace{6cm}=\phi\left(F-k_2F^*\right)+\mu_{A,2}k_2A^*(1-k_2)\\
	\hspace{6cm}\leq0 \quad\mbox{in}\quad \Gamma_{\rr>1}, ~ \mbox{so}~ A\leq k_2A^*.\\
	\end{array}
	$$
\end{enumerate}
This ends the proof of the Lemma.

\end{document}